\documentclass{amsart}
\copyrightinfo{2014}{the authors}
\pdfoutput=1
\usepackage{amsmath}
\usepackage{amsfonts}
\usepackage{amsthm}
\usepackage{amssymb}
\usepackage{graphicx, url, todonotes}

\newtheorem{theorem}{Theorem}[section]

\theoremstyle{definition}
\newtheorem{definition}[theorem]{Definition}

\begin{document}

\title{Symmetric Colorings of Polypolyhedra}

\author[belcastro]{sarah-marie belcastro}
 \address{belcastro: MathILy and Smith College, MA, USA}
\email{smbelcas@toroidalsnark.net}

\author[Hull]{Thomas C. Hull}
\thanks{The authors would like to thank Robert Lang for sharing the polypolyhedra {\em Mathematica} code that allowed us to produce the illustrations in this paper.  The final version of this paper will appear in {\em Origami$^6$} published by the AMS in 2016.}
 \address{Hull: Department of Mathematics, Western New England U., Springfield MA, USA}
\email{thull@wne.edu}



\maketitle

%

%
%

\section{Introduction}


Most origami practitioners are familiar with Hull's Five Intersecting Tetrahedra model \cite{H1}, but fewer are familiar with Lang's 54 models that generalize this work \cite{L1}.  Lang terms these \emph{polypolyhedra} because they are compounds of 1-skeleta (so there are many, or poly-, of them) that have polyhedral symmetries.  Polypolyhedra are usually depicted with each component in a single color.  Alternate colorings of polypolyhedra can highlight various, sometimes hidden, structures and symmetries related to the underlying polyhedral rotational symmetry group.  Hull created a symmetric coloring of the Five Intersecting Tetrahedra that used 6 colors, each of which appeared on each of the five tetrahedra, and asked how many such colorings there are of the compound model.  We answered this question (which turns out to be of primarily mathematical interest because the different colorings look quite similar to each other); then, belcastro wondered whether our analysis might extend to producing and counting symmetric colorings of Lang's polypolyhedra. The present work represents our findings.

\section{Summary of Lang's Work}

What exactly are the objects that we are coloring?  They are nontrivial compounds of 1-skeleta, with the following properties \cite{L1}.
\begin{itemize}
\item Each vertex has degree at least two.  This excludes objects that are secretly piles of sticks. 
\item No two edges intersect except at vertices.  This excludes objects such as the stella octangula (compound of two tetrahedra) and ensures that each edge may be made from a separate unit of paper.
\item The compound is edge transitive;  there exists a rotational symmetry of the compound mapping any given edge to any other given edge.  Alternatively, there is only a single orbit of edges under the symmetry group action.  
\item Only tetrahedral (order 12), cuboctahedral (order 24), and icosidodecahedral (order 60) rotational symmetry groups are allowed.  
\end{itemize}
It follows from edge-transitivity that there are at most two vertex orbits under the symmetry group action.  Consider a sequence of group elements $g_1,\dots,g_n$ acting on edge $e$ to produce all $n$ edges $g_1(e)=e_1,\dots, g_n(e)=e_n$.  Examine the orbits of the vertices $v_1, v_2$ of $e$ under this sequence of actions; there are at most two distinct orbits generated.

Conspicuously missing from the list of properties is that the 1-skeleta are of 2- or 3-dimensional polytopes.  This is because several of the 1-skeleta are nonconvex and many are not the skeleta of polytopes because it is not possible to induce 2-faces from the 1-skeleta.  (These have skew polygon cycles bounding the areas that would visually indicate 2-faces.)

Lang used a comprehensive computer search to determine all possible types of polypolyhedra.  Some types appear in several variants.
 Here is a straightforward way to think of the variants on a polypolyhedron type: View each component as a linkage, with hinges at the vertices and rods for edges.  Now take the (radially) outermost vertices and push them towards the center.  At some point, edges of different components will intersect; after passing through each other, the configuration has a different interlacing, and the set of these interlacings comprises the set of variants.
 
 We list Lang's polypolyhedra here with various symmetry properties noted, and list first those that have non-polygon components.

 \begin{center}
 \begin{tabular}{|p{4.4cm}|l|c|c|}
\hline
Name/Description & symmetry group & vertex transitivity & variants \\ \hline\hline
Five Intersecting Tetrahedra & icosidodecahedral & vertex transitive & n/a \\\hline
 five intersecting non-convex hexahedra
 & icosidodecahedral & not vertex transitive  & n/a \\ \hline
four intersecting bi-3-pyramids (no base edges)
& cuboctahedral & not vertex transitive  & 2 \\ \hline
five intersecting edge-dented tetrahedra 
& icosidodecahedral & not vertex transitive  & 2 \\ \hline
ten intersecting bi-3-pyramids (no base edges)
& icosidodecahedral & not vertex transitive  & 3 \\ \hline
six intersecting bi-5-pyramids (no base edges)
& icosidodecahedral & not vertex transitive  & 4 \\ \hline
\hline \hline
 four interlaced triangles & cuboctahedral & vertex transitive & n/a \\ \hline
six interlaced pentagons & icosidodecahedral & vertex transitive & n/a \\ \hline
ten interlaced triangles & icosidodecahedral & vertex transitive & n/a \\ \hline
three interlaced squares & tetrahedral & not vertex transitive  & n/a \\ \hline
four interlaced hexagons & cuboctahedral & not vertex transitive  & n/a \\ \hline
six interlaced decagons & icosidodecahedral & not vertex transitive  & n/a \\ \hline
 ten interlaced hexagons & icosidodecahedral & not vertex transitive  & n/a \\ \hline
 four interlaced triangles & tetrahedral & vertex transitive & n/a \\ \hline
six interlaced squares & cuboctahedral & vertex transitive & 2 \\ \hline
eight interlaced triangles & cuboctahedral & vertex transitive & 3 \\ \hline
 twelve interlaced pentagons & icosidodecahedral & vertex transitive & 5 \\ \hline
twenty interlaced triangles & icosidodecahedral & vertex transitive & 23 \\
\hline
\end{tabular}
\end{center}

\section{Counting Symmetric Colorings of Polypolyhedra}\label{sec:thing}

We are interested in edge colorings of polypolyhedra that respect their underlying symmetry groups. 

\begin{definition}\label{def:symm}
We call an edge coloring $c$ of a polypolyhedron $P$ a {\em symmetric coloring} if the action of any element of the underlying symmetry group either leaves all edges of a given color the same color, or takes all edges of a given color to the set of edges of another color.  That is, for the edges $e_i$ in a fixed color class and $g$ any element of the symmetry group, we have either $c(g(e_i)) = c(e_i)$ for all $i$, or $c(g(e_i)) = c(g(e_j))$ exactly when $c(e_i)=c(e_j)$.
\end{definition}

Because any action of the symmetry group takes one component of a polypolyhedron to another component, every component must use the same number of colors; thus the number of colors must divide the number of edges in a component.    

There always exists the symmetric coloring in which all edges of a component are the same color, and there is exactly one such coloring.  At the other end of the spectrum, {\em rainbow colorings} have each edge of a component a different color.  

We now carefully describe a particular type of symmetric coloring of the Five Intersecting Tetrahedra, and note that this subsection serves as a model for how arguments will proceed in the remainder of this section.

\subsection{Five Intersecting Tetrahedra (FIT)} 

The FIT has icosidodecahedral rotational symmetry (group $A_5$) and thus may be inscribed in a dodecahedron, as shown in Figure \ref{fig:inscr}(left).  Every element of $A_5$ has order 1, 2, 3, or 5, corresponding to the identity and 2-fold, 3-fold, and 5-fold rotations around axes passing through antipodal dodecahedral edges, vertices, and faces respectively.
We may decompose the dodecahedron's edges into six non-perfect matchings; see Figure \ref{fig:inscr}(center).  The dodecahedron is trivially a polypolyhedron, and one can see by inspection that the edge-coloring given by the matching decomposition is a symmetric coloring.

\begin{definition}\label{def:band}
Given an edge $e_i$ of the FIT considered as a vector in $\mathbb R^3$, find the 5-fold rotational axis $a$ such that $e_i\cdot a$ is minimal.  Consider any element $r_a\in A_5$ corresponding to $a$.  The \emph{band} generated by $e_i$ is the orbit of $e_i$ under $r_a$. 
\end{definition}

 In Figure \ref{fig:inscr}(right), we see that the edges that form a band in the FIT correspond to one of the six exhibited non-perfect matchings of the dodecahedron.   
\begin{figure}[htp]
\begin{center}
	\includegraphics[scale=.45]{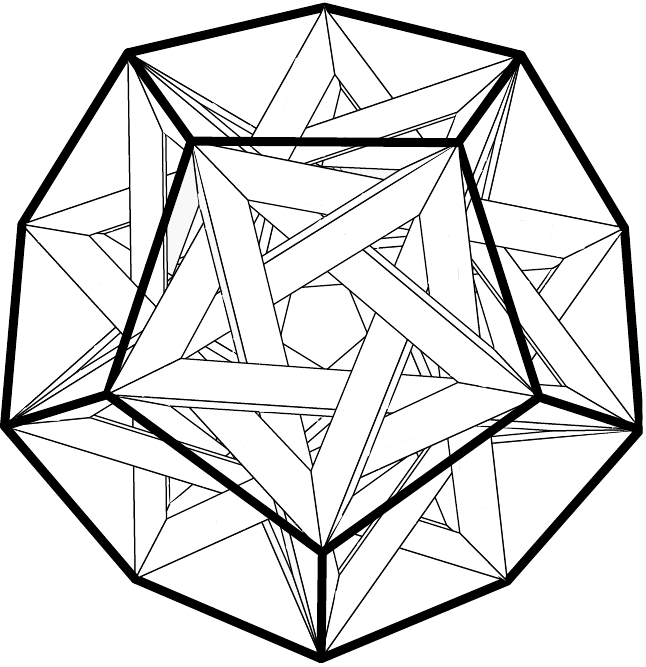}\hfill \includegraphics[scale=.45]{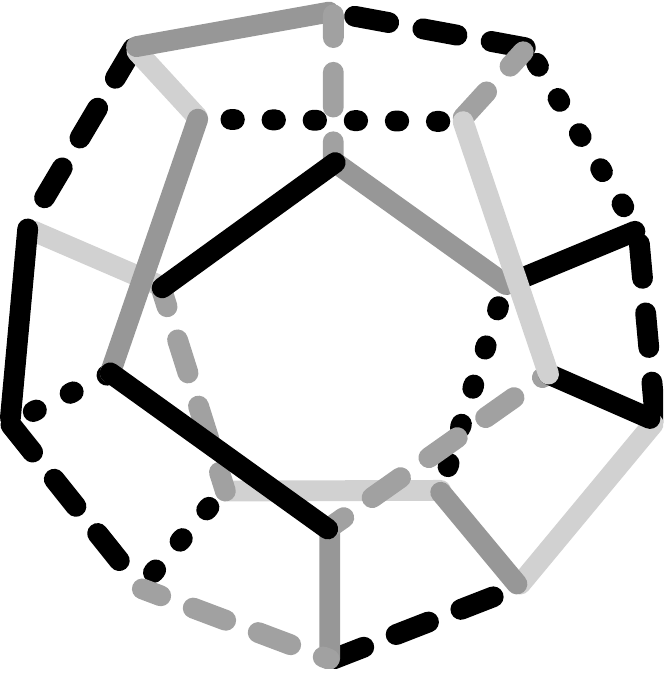}\hfill \includegraphics[scale=.45]{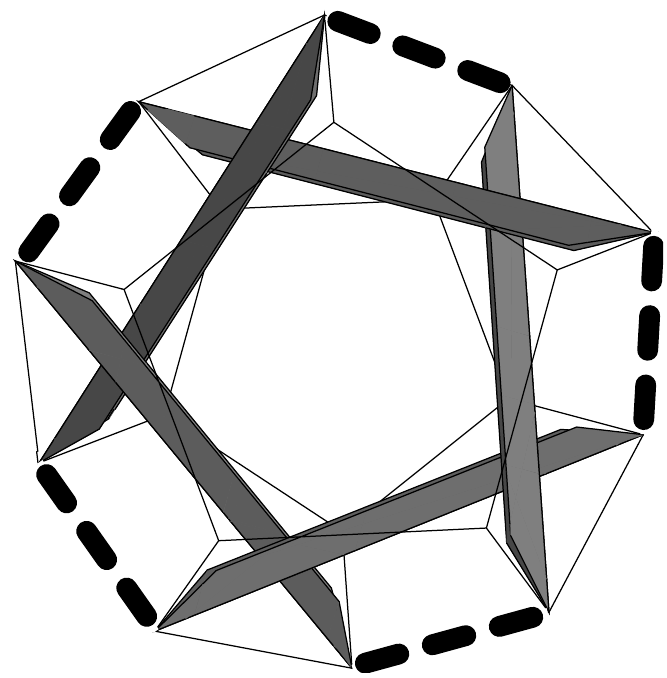}
	\caption{The FIT inscribed in a dodecahedron (left) and the dodecahedron decomposed into matchings (center), and matching with corresponding band highlighted (right).}
	\label{fig:inscr}
\end{center}
\end{figure}
 Indeed, these non-perfect matchings can be defined in exactly the same way as bands, with corresponding matchings and bands sharing the same 5-fold rotational axis.  We therefore refer to the matching corresponding to a band $B$ as $M(B)$. 

We may color the edges of the FIT in bands by repeatedly adding a color to an uncolored edge and generating the band corresponding to this edge. We refer to this FIT coloring as a \emph{band coloring}.  Because the bands are in correspondence with the dodecahedral matchings, and these matchings form a symmetric coloring, a band coloring is also a symmetric coloring.
 
Note that an individual tetrahedron has no 5-fold rotational symmetries, and so no two edges of a tetrahedron can belong to the same band; thus, a band coloring necessarily induces a rainbow coloring on each tetrahedron of an FIT.  A band coloring is shown in Figure \ref{fig:fitband}. Origami artist Denver Lawson also discovered band colorings of the FIT and displayed such a coloring at the 2012 Fall Convention of the British Origami Society. 
\begin{figure}[htp]
\begin{center}
\vspace{-.1in}
	\includegraphics[scale=.325]{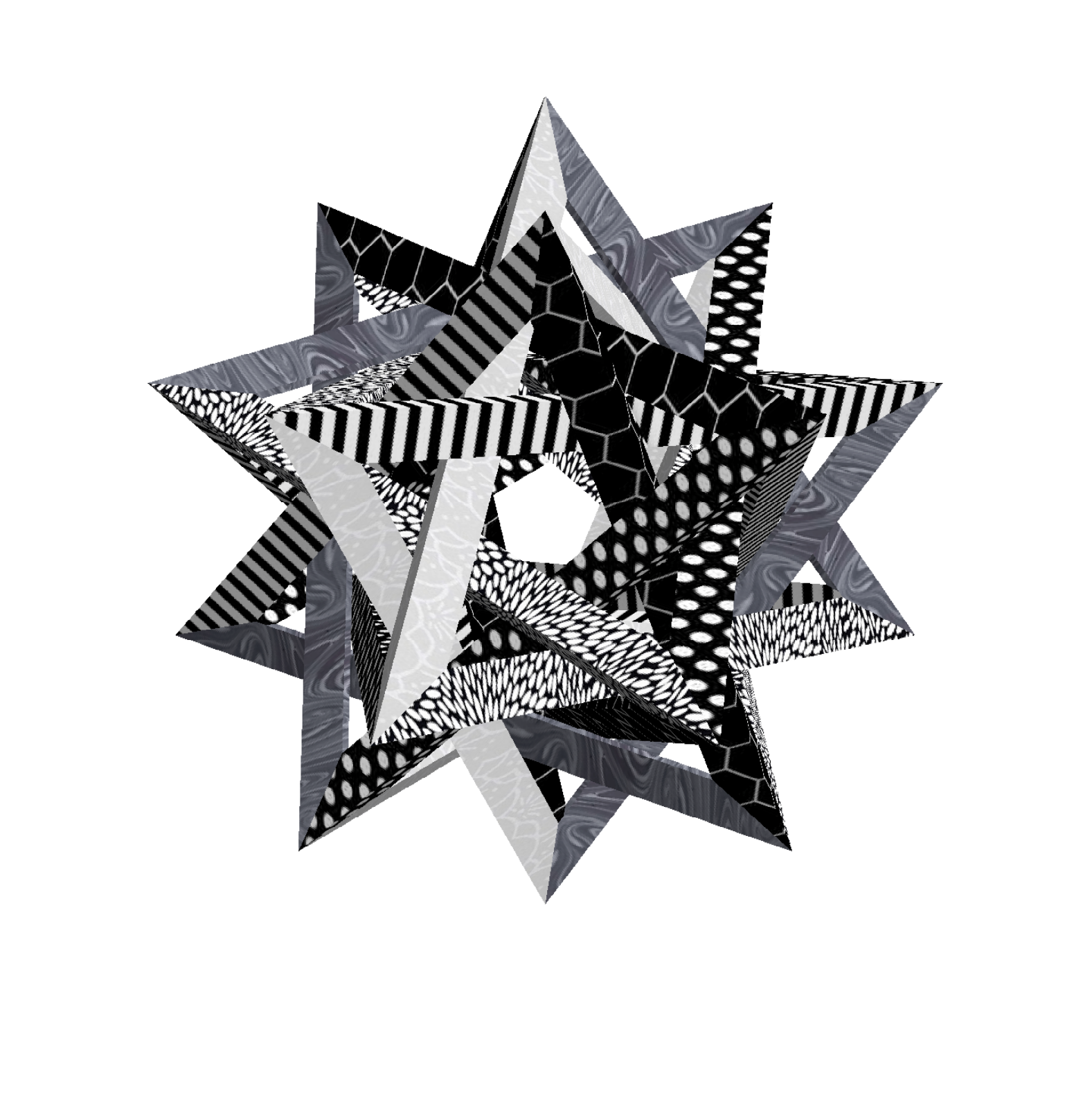}\hspace{.5cm}
	\includegraphics[scale=.3]{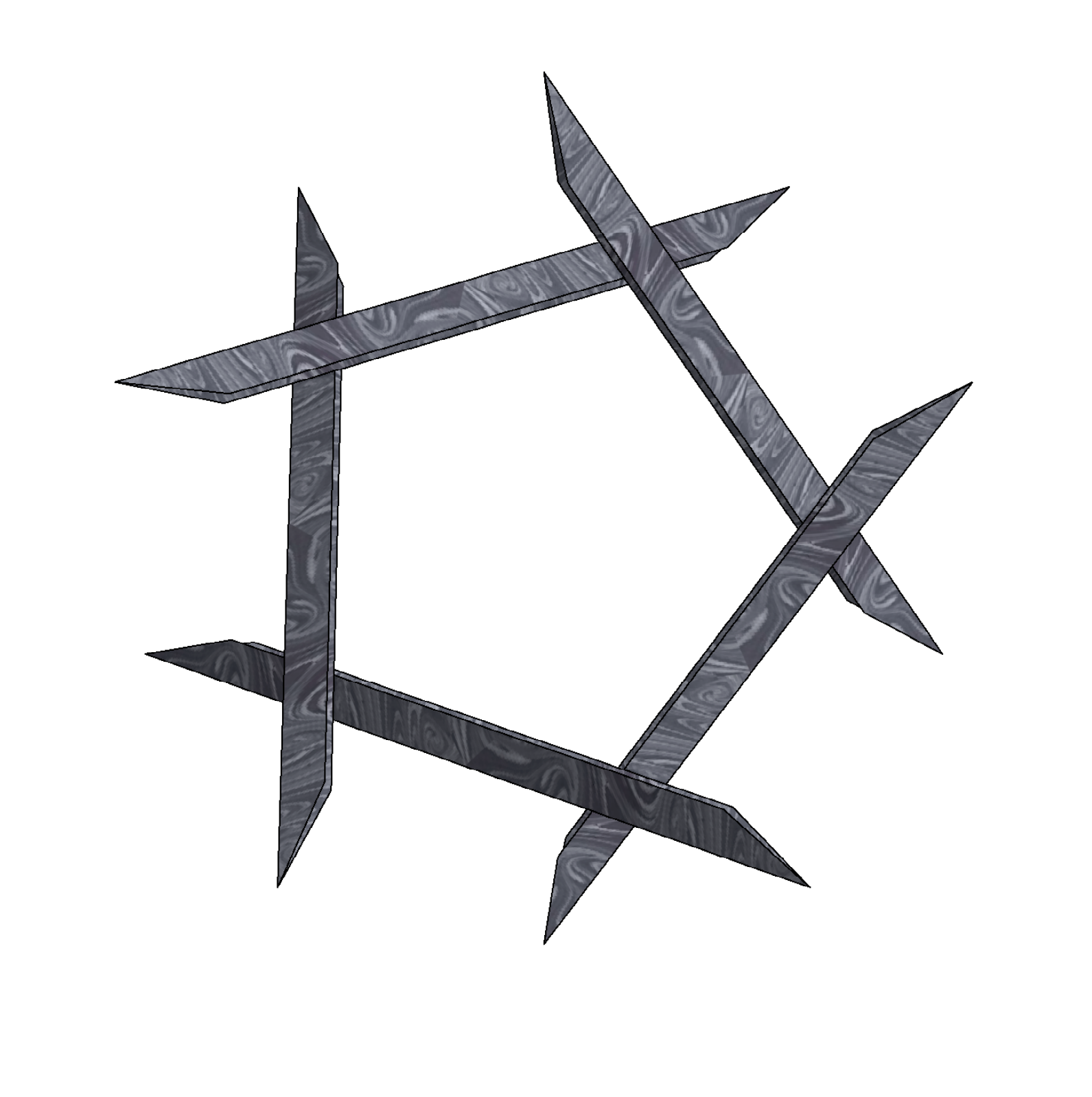}
	\vspace{-.4in}
	\caption{ A symmetrically colored FIT (left) with individual band forming a pentagonal circuit (right).}
	\label{fig:fitband}
\end{center}
\end{figure}




\subsubsection{Counting Band Colorings of the FIT}

First note that there are 60 different colorings for a rainbow edge-colored tetrahedron. Here is why: Fix the color of edge $e_1$, and note that there is one edge $e_2$ that is independent of $e_1$. There are then 5 color choices for $e_2$, and $4!$ color choices for the remaining 4 edges.  However, this total of 120 overcounts by a factor of two because of the 2-fold rotational symmetry that fixes $e_1$ and $e_2$.

Second, note that while the FIT has symmetry group $A_5$, a single tetrahedron has symmetry group $A_4$.

\begin{theorem}
There are exactly 12 different band colorings of the FIT.
\end{theorem}


\begin{proof}[Proof 1 (tetrahedral)]\label{tetraproof} Edge-color one of the five tetrahedra, $T$, with one of the 60 possible rainbow edge-colorings.  Consider an edge $e_1$ of $T$; it is part of a band $\{r_a^k(e_1)\}$, which determines an edge on each of the other four tetrahedra.  Now consider a second edge $e_2$ of $T$ with corresponding band $\{r_b^k(e_2)\}$. There is exactly one element $g\in A_4$ (and thus in $A_5$) such that $g(e_1)=e_2$, and thus $g(\{r_a^k(e_1)\}) =  \{r_b^k(e_2)\}$ because the symmetries permute the bands.  That is, the position of the $e_2$ band is determined, and by the same reasoning so are the other bands.  This produces 60 band colorings of the FIT, but we could have used any of the five tetrahedra as $T$. Thus there are $60/5 = 12$ different FIT band colorings.
\end{proof}

\begin{proof}[Proof 2 (band)]\label{bandproof} Fix a band, $B_1=\{r_a^k(e_1)\}$, and assign it a color.  Fix a second band $B_2=\{r_b^k(e_2)\}$; there are 5 colors that can be assigned to this band, but we could have chosen any of the 5 remaining bands to be $B_2$, so this does not affect the total number of colorings.   
There are $4!$ ways to color the remaining 4 bands.  However, $M(B_1)$ contains a dodecahedron edge opposite to an edge of $M(B_2)$ (see Figure \ref{fig:inscr}(center)). Therefore there exists $g\in A_5$ of order 2 such that  $B_1, B_2$ are both invariant under action of $g$. Thus, there are $4!/2 = 12$ different ways to color the remaining bands and 12 different band colorings of the FIT.
\end{proof}


\begin{proof}[Proof 3 (Burnside)]\label{burnproof}  Burnside's Lemma states that if $G$ acts on $S$, then the number of orbits of the action is $\displaystyle\frac{1}{\vert G\vert}\sum_{g\in G}\vert\{s\in S\ \vert \ g(s) = s \}\vert$.  Here, $G = A_5$ and $S$ is the set of $6!=720$ possible band colorings. No element of $A_5$ fixes any element of $S$, except for the identity $e\in A_5$ which fixes all $s\in S$. Thus there are $720/60=12$ orbits of $A_5$ on $S$.  Each orbit contains colorings that are equivalent under some symmetry in $A_5$, and thus there are 12 distinct FIT band colorings.
\end{proof}

\medskip

\subsection{Generalizations from the FIT}

The FIT is a remarkably symmetric polypolyhedron, even among polypolyhedra.  Each band has two properties: (a) it corresponds to a matching in the dodecahedron 
graph, and (b) the edge units in the band sequentially touch each other and form a visual cycle around the model.  As we will see, these two properties do not coincide for sets of edge units in other polypolyhedra. Thus, we investigate two types of symmetric polypolyhedra edge colorings, \emph{matching} edge colorings and \emph{visual band} edge colorings.  

\subsubsection{Matching Colorings}\label{sec:MC}

A matching coloring descends directly from the underlying symmetry group, so matching colorings always exist and counting such colorings will proceed analogously to the proofs given for the FIT above.  In this way, we only need to make a 1-1 correspondence between polyhedral matchings and sets of polypolyhedral edges in order to quickly know the number of matching edge colorings.  

Analogous to our matching decomposition of the dodecahedron, there is a decomposition of the cube (rotational symmetry group $S_4$) into 4 non-perfect matchings, each with 3 edges; see Figure \ref{fig:matches}(left). 
\begin{figure}[htp]
\centerline{ \includegraphics[scale=.35]{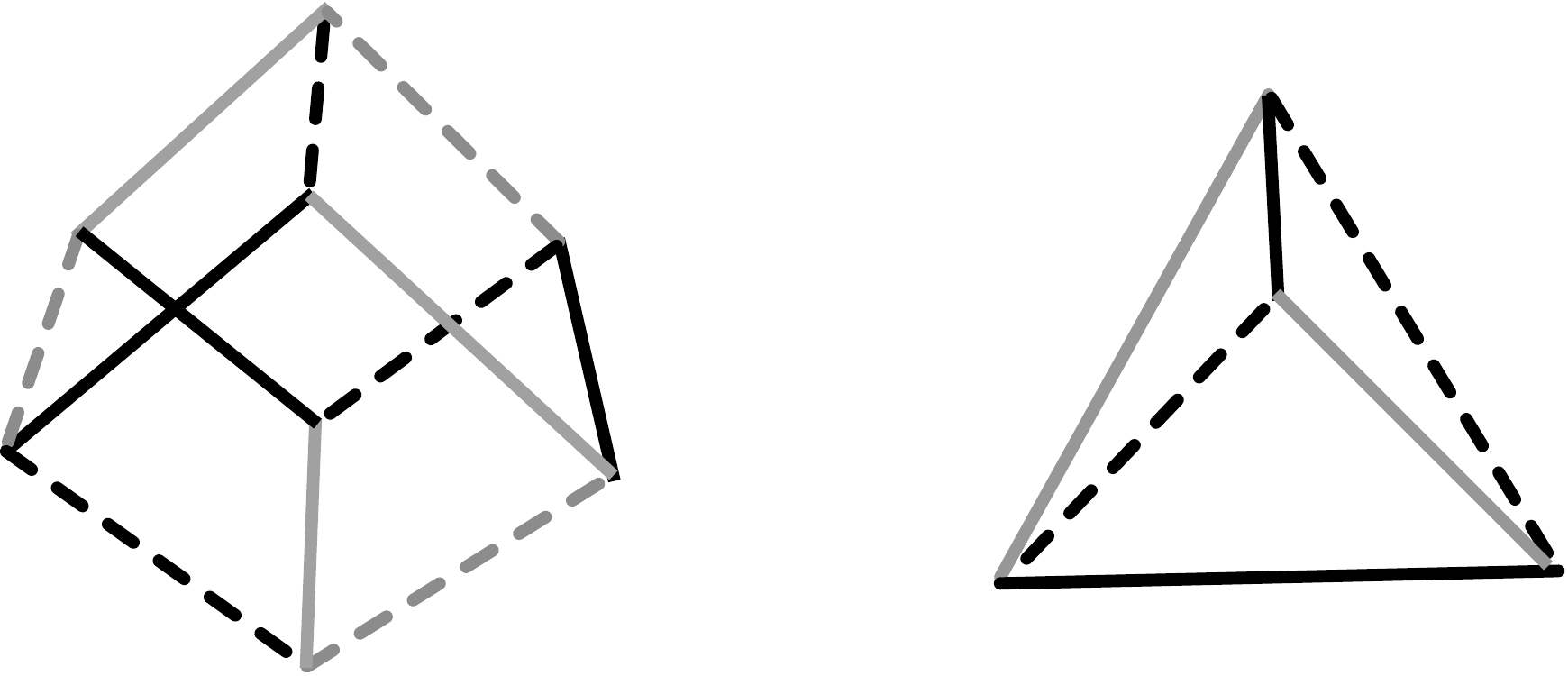}}
\vspace{-.2in}
\caption{The cube and tetrahedron with matchings highlighted.}
\label{fig:matches}
\end{figure}
Coloring this configuration properly requires four colors.
We will compute the number of different colorings as Proof 2 (bands) above.  Fix a matching $M_1$ and assign it a color. Choose a second matching $M_2$; assignment of color to $M_2$ does not affect the total number of colorings.  There are $2!$ ways to color the other two matchings.  Because $M_1$ contains an edge opposite to an edge of $M_2$, there exists $g\in S_4$ of order 2 such that $M_1$ and $M_2$ are both invariant under action of $g$.  Thus there is only $2!/2=1$ way to color these matchings of the cube.
Alternatively, we could use Burnside's Lemma.  Only the identity in $S_4$ leaves all $4!=24$ colorings of the matchings invariant.  Thus there is only $24/24 = 1$ way to color such matchings.



The self-dual tetrahedron \emph{only} decomposes into three perfect matchings, each with two edges; see Figure \ref{fig:matches}(right).  There are $3!=6$ possible ways to color these matchings.  All of these colorings remain fixed under the identity transformation and the three 2-fold symmetries of the tetrahedron.  The tetrahedral rotational symmetry group $A_4$ has order 12, and thus there are $(6+3\cdot 6)/12= 2$ different ways to color these matchings.

\subsubsection{Visual Band Colorings}\label{subsec:visual bands}
A visual band must have struts touching sequentially, so existence of a visual band coloring depends on the incidence relations between the polypolyhedron struts.  Thus, identifying and counting visual band colorings is somewhat idiosyncratic.  In particular,  the existence of visual band colorings varies depending on which variant of a polypolyhedron we consider.  Moreover, we will often encounter non-proper visual band colorings.  Consider, for example, a polypolyhedron with polygon components.  Any visual band will be composed of edges from a single component, and this induces the unique monochromatic-component coloring.

%
%

\subsection{Five Intersecting Edge-dented Tetrahedra (FIET)} On each component of the FIET, the vertices form 3-pyramids (without base edges), so there are 4 ``faces," each of which has 6 edges.  The FIET has icosidodecahedral symmetry, and there are two polypolyhedral variants. (It is denoted 5-4-6 in Lang's nomenclature.) The two polypolyhedral variants are substantially different: the components in one (A, see Figure \ref{fig:fietA}) look like in-dented tetrahedra, while the components in the other (B, see Figure \ref{fig:fietB}) correspond to severely edge-out-dented tetrahedra.
\begin{figure}[htp]
\centerline{ \includegraphics[scale=.45]{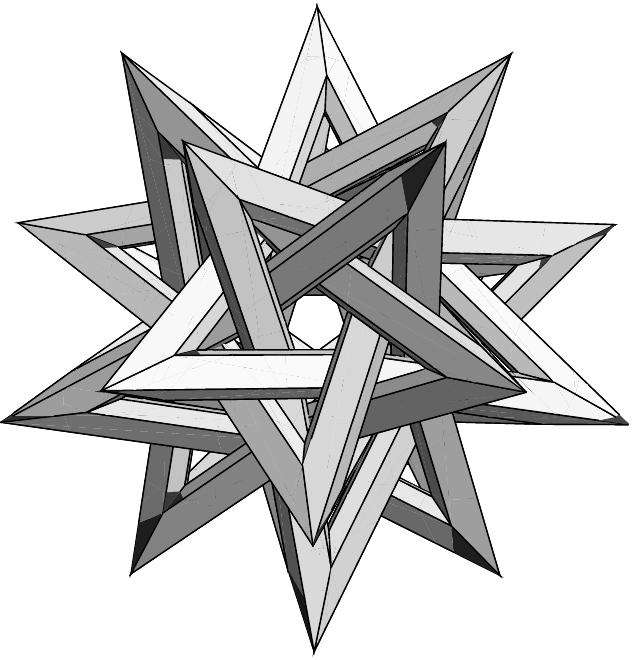} \includegraphics[scale=.4]{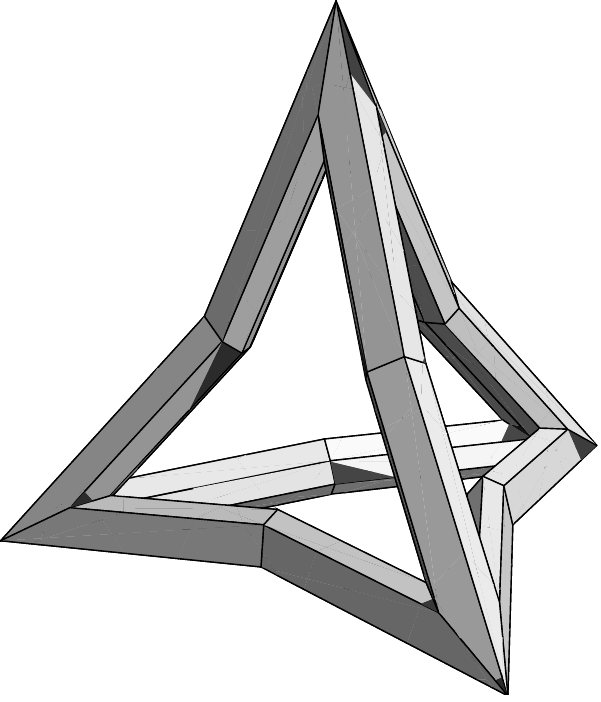}  \includegraphics[scale=.35]{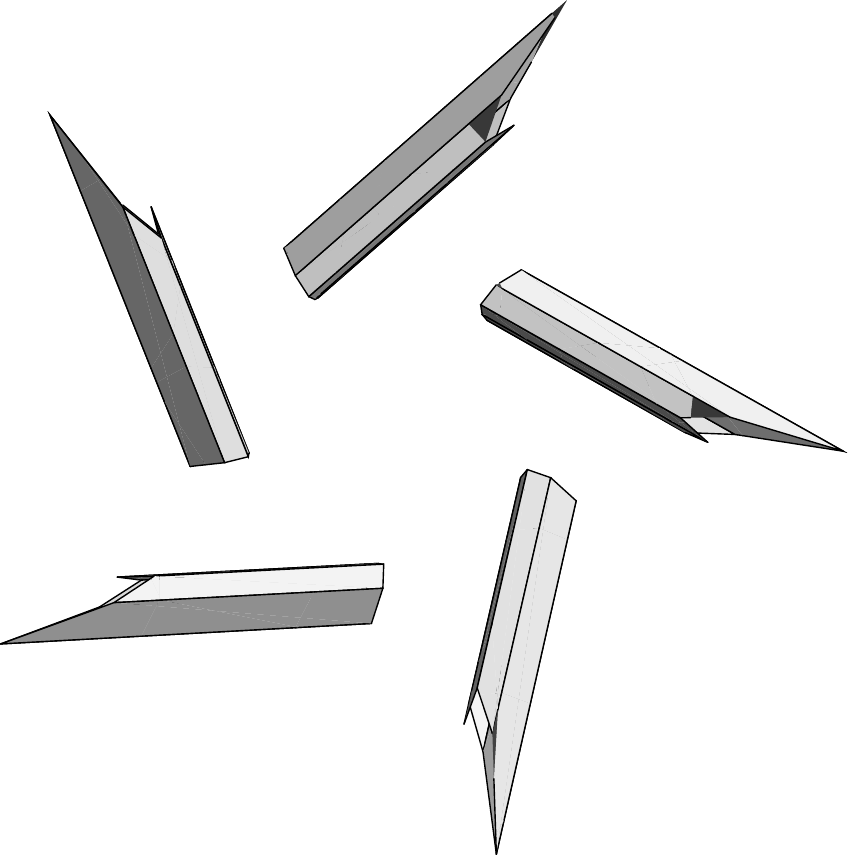}\includegraphics[scale=.3]{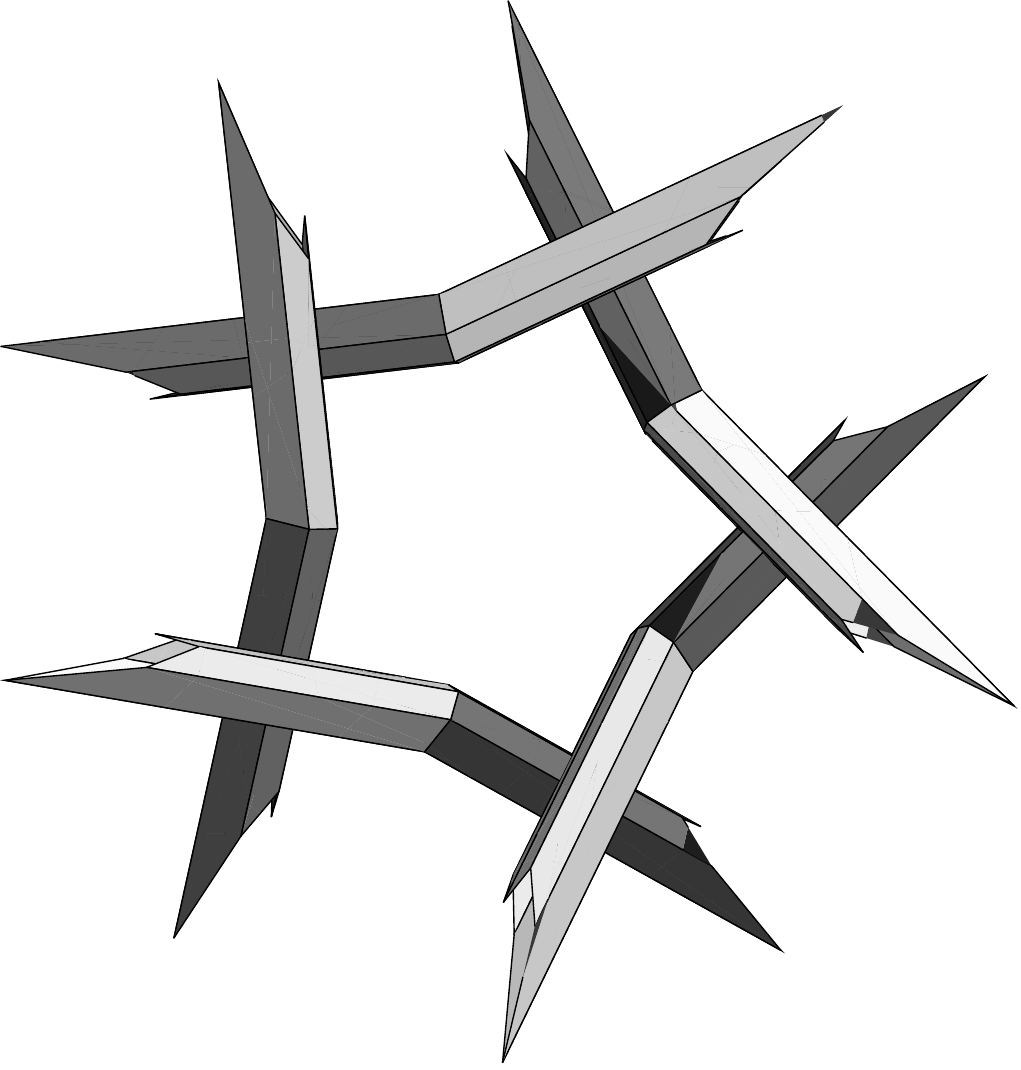} }
\caption{FIET variant A (left), with a single component (center), a single matching (right), and a pair of matchings (right-er).}
\label{fig:fietA}
\end{figure}

Variant A has a matching coloring but not a visual band coloring.  Taking one edge from each component, we form a dodecahedral matching $M$.  However, there is a second matching $M'$ of the FIET corresponding to this same dodecahedral matching (see Figure \ref{fig:fietA}(right-er)), and a total of 12 matchings.  
Given a coloring of the six dodecahedron matchings (one of the 12 that exist), there are ${12\choose 6}$ ways to assign 6 (of 12) colors to the six matchings $M_1,..., M_6$.   Then there are $6!$ ways to assign the remaining colors to the matchings $M_1',...,M_6'$.  Thus, there are $12 \cdot {12\choose 6} \cdot 6!$ colorings of this configuration of polypolyhedral matchings.


For variant B, a visual band $B$ looks like a 5-pointed star (see Figure \ref{fig:fietB}(right)), with one peak (two edges) from each component. It corresponds to a dodecahedral matching $M(B)$, so a visual band coloring uses
\begin{figure}[htp]
\centerline{ \includegraphics[scale=.45]{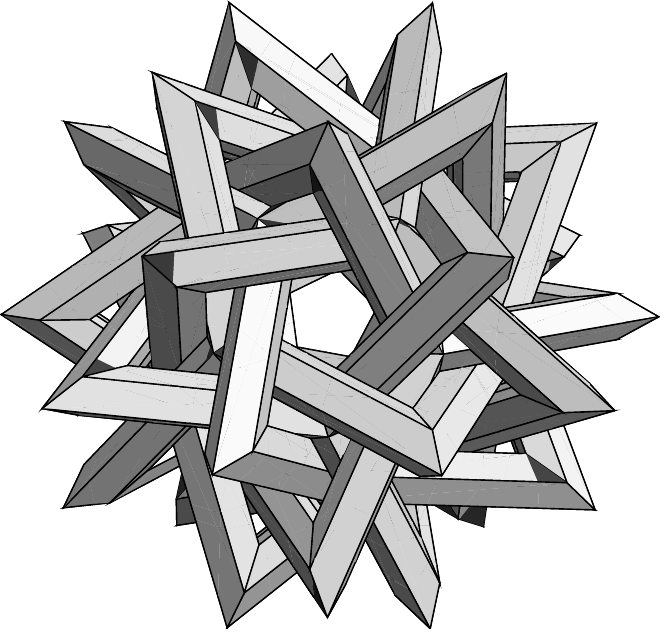} \includegraphics[scale=.4]{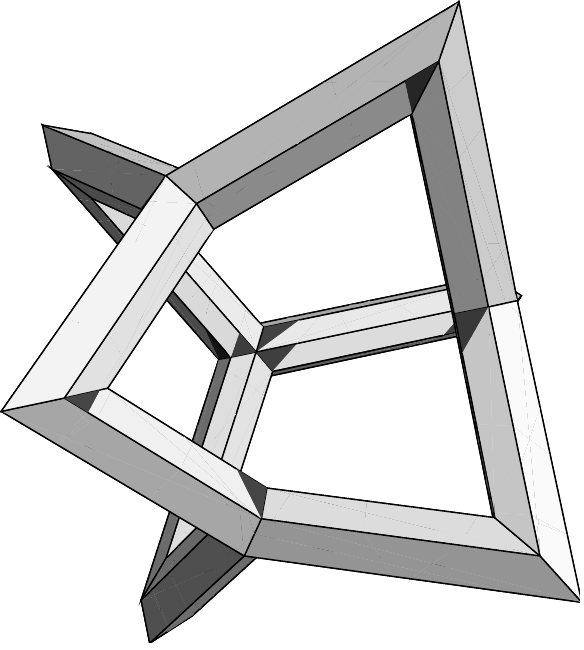}  \includegraphics[scale=.42]{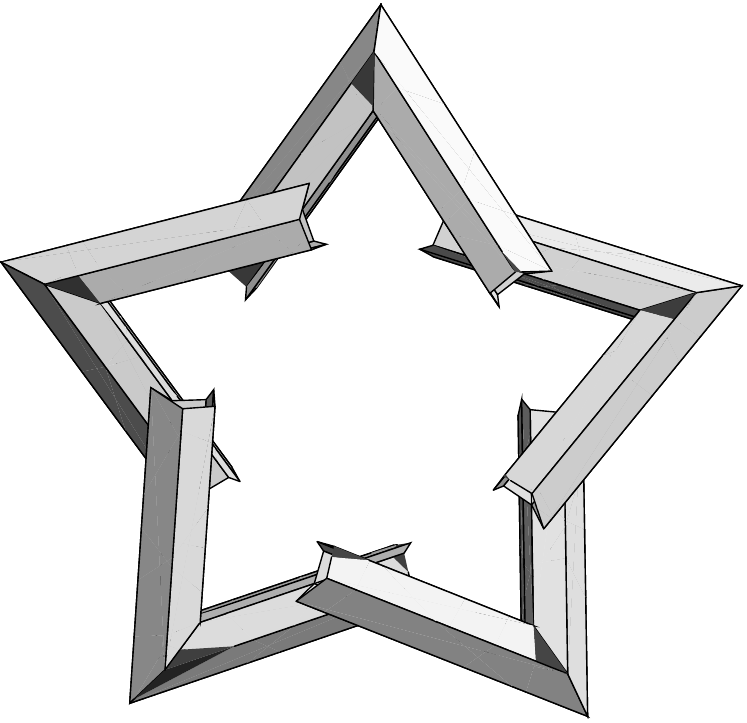}\includegraphics[scale=.37]{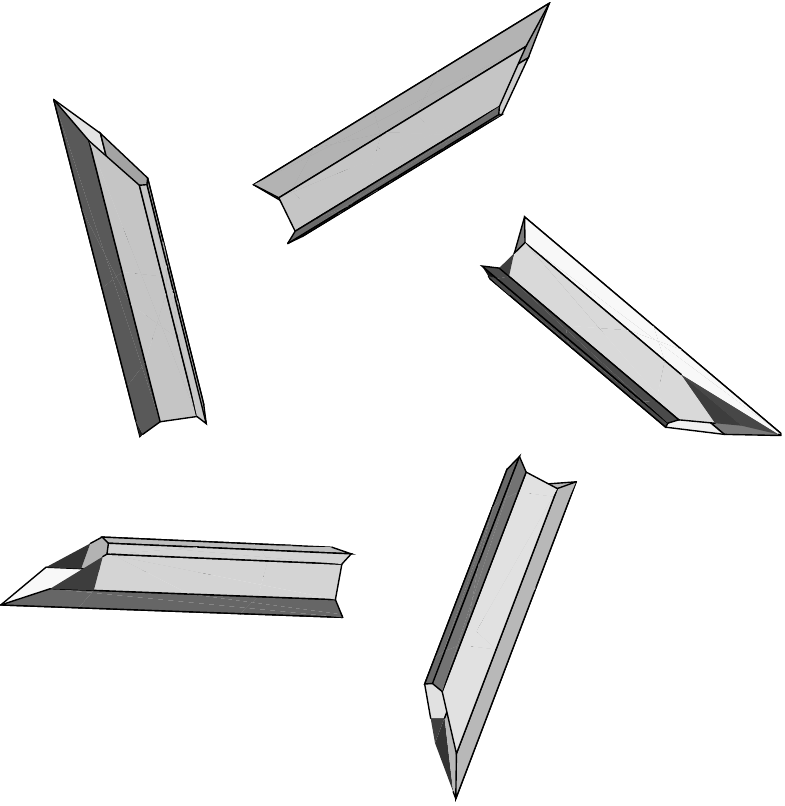} }
\caption{The FIET variant B (left), with a single component (center), a visual band (right), and a matching (right-er).}
\label{fig:fietB}
\end{figure}
6 colors per component, and there are 12 such colorings. 
The matching colorings of variant B are the same as in A.


\subsection{Five Intersecting Non-convex Hexahedra (FINH)} Each component of the FINH has skew 4-sided faces (see Figure \ref{fig:finh}(center)).  The FINH has icosidodecahedral symmetry.  (It is denoted 5-6-4 in Lang's nomenclature.)
\begin{figure}[htp]
\centerline{ \includegraphics[scale=.45]{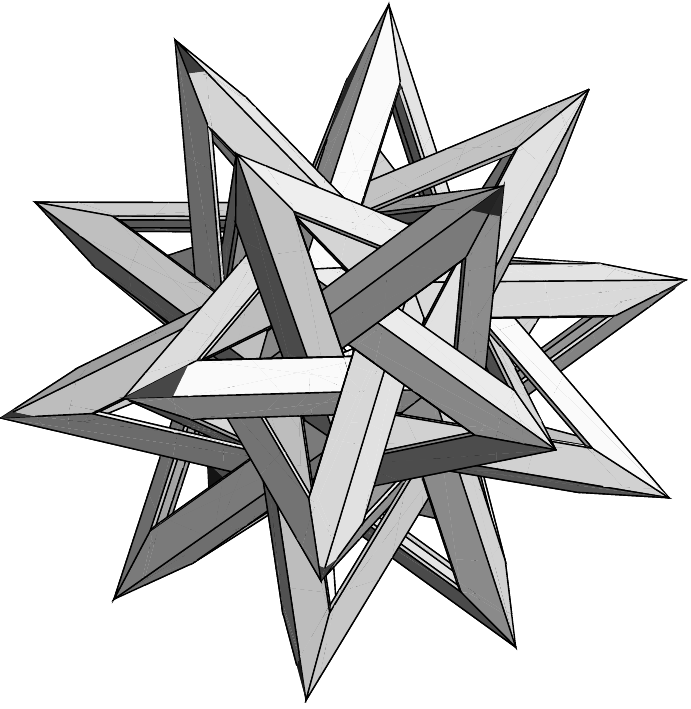} \includegraphics[scale=.35]{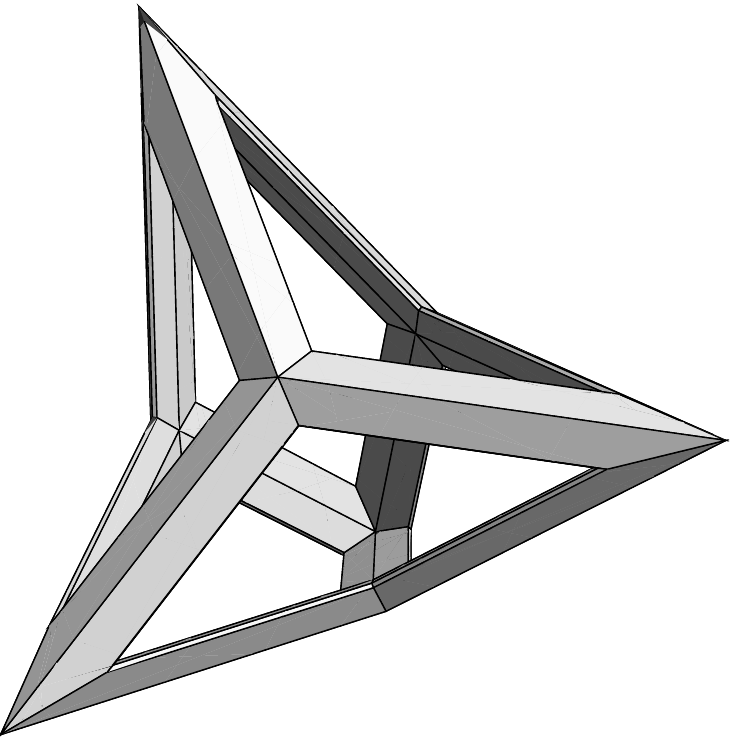}  \includegraphics[scale=.35]{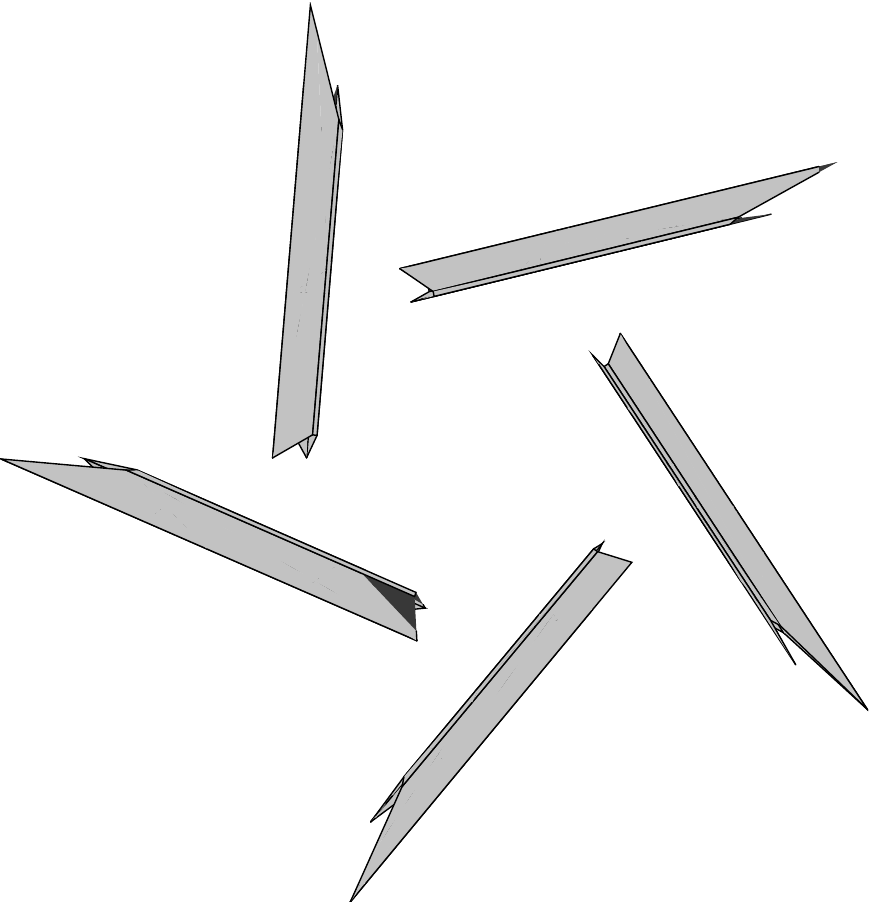}\includegraphics[scale=.27]{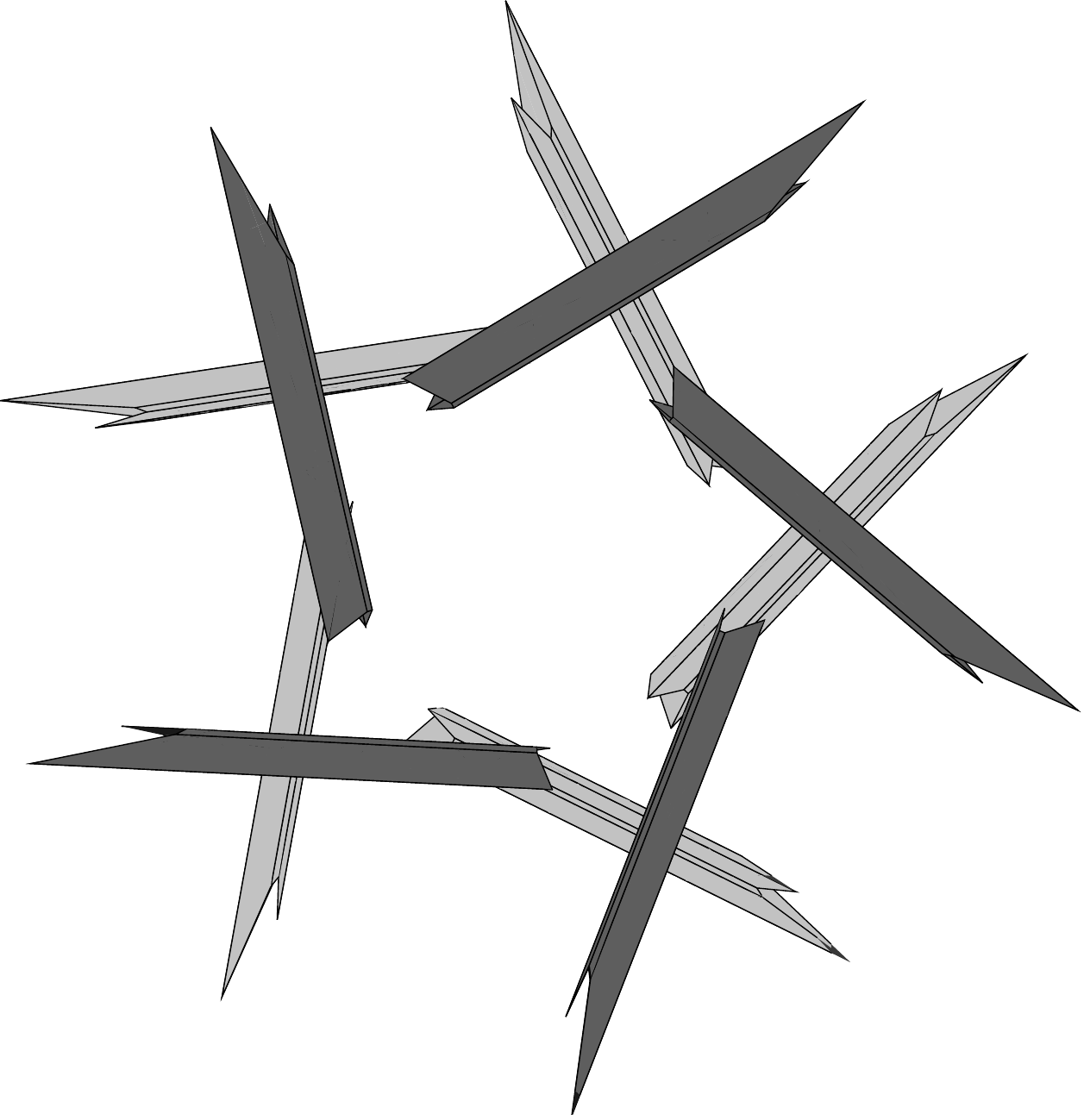} }
\caption{The FINH (left), with a single component (center), a single matching (right), and a pair of matchings (right-er).}
\label{fig:finh}
\end{figure}
The matchings are identical to those of FIET variant A, and so are the combinatorics.  There are no visual bands.

\subsection{Six Intersecting Bi-5-pyramids without base edges (SIB5P)} Each component of the SIB5P has five 4-sided ``faces." The SIB5P has icosidodecahedral symmetry, and there are four polypolyhedral variants (denoted 6-5-4 in Lang's nomenclature).
 The four polypolyhedral variants form two pairs that are substantially different: the components in one pair (A1 and A2, see Figure \ref{fig:sib5p} (left, center)) have 5-valent vertices visually prominent, while the components in the other (B1 and B2, see Figure \ref{fig:sib5p}(right, right-er) ) have 2-valent vertices visually prominent.
\begin{figure}[htp]
\centerline{ \includegraphics[scale=.45]{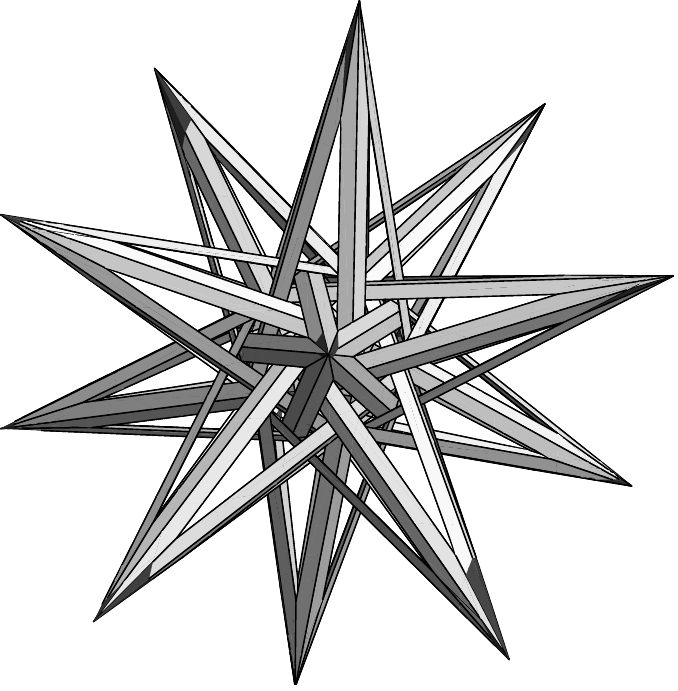} \includegraphics[scale=.45]{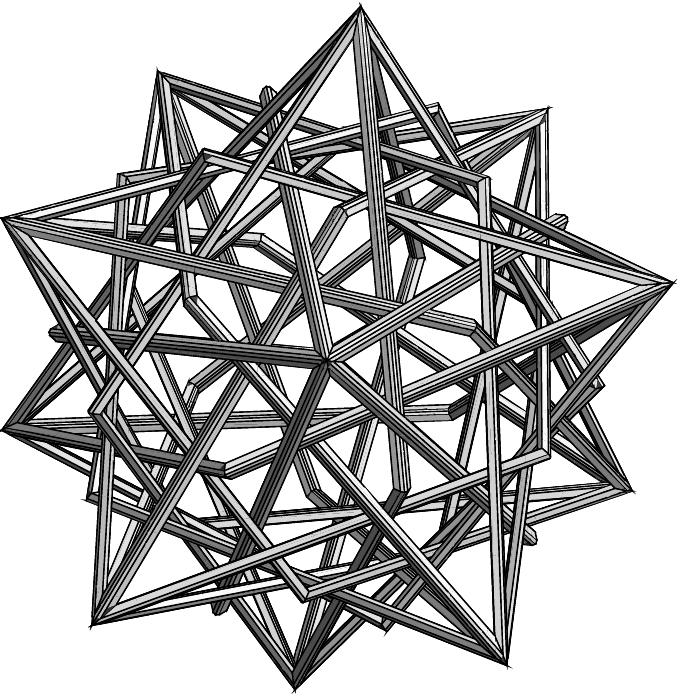}  \includegraphics[scale=.45]{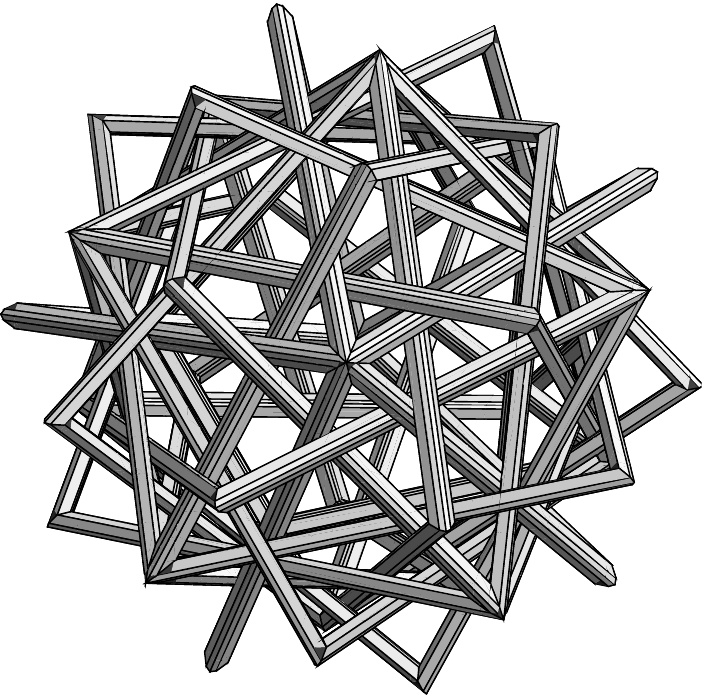}\includegraphics[scale=.45]{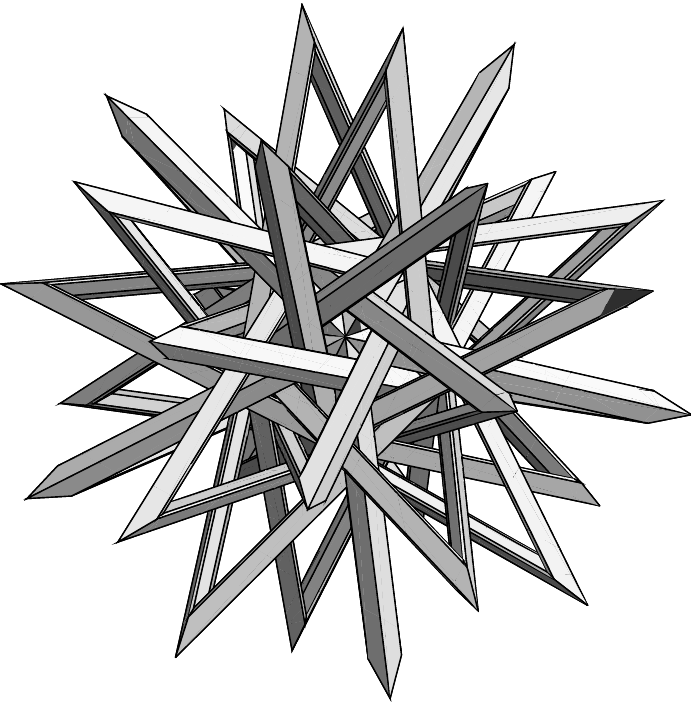} }
\caption{The SIB5P variant A1 (left), variant A2  (center), variant B1 (right), and variant B2 (right-er). }
\label{fig:sib5p}
\end{figure}
However, the combinatorics of coloring the four variants are all the same.

There is no visual band coloring of an A1 or A2 variant; a visual band would need to use 2 of the 5 edges at a vertex, and as 2 does not divide 5, this would not extend to a symmetric coloring.  Consider instead variants B1, B2. While there is a visual 5-pointed star formed by pairs of edges incident to 2-valent vertices, the pairs are not incident and thus do not form a visual band (see Figure  \ref{fig:sib5pM}).
\begin{figure}[htp]
\centerline{ \includegraphics[scale=.4]{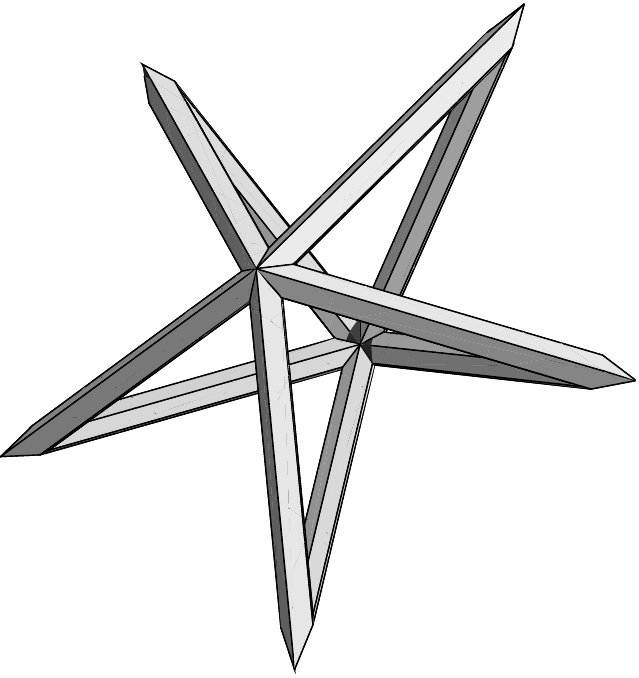} \includegraphics[scale=.36]{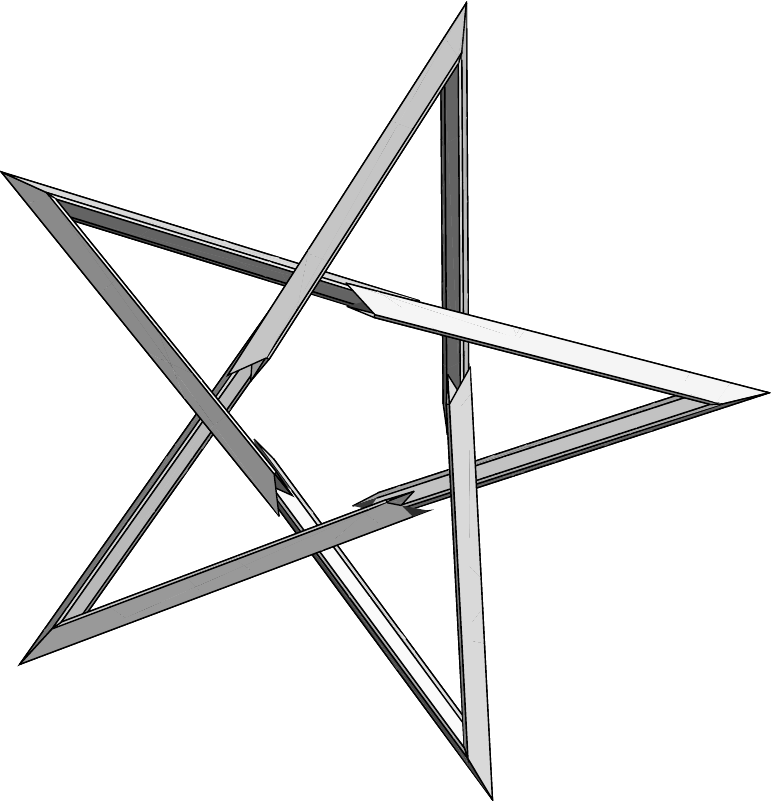}  \includegraphics[scale=.33]{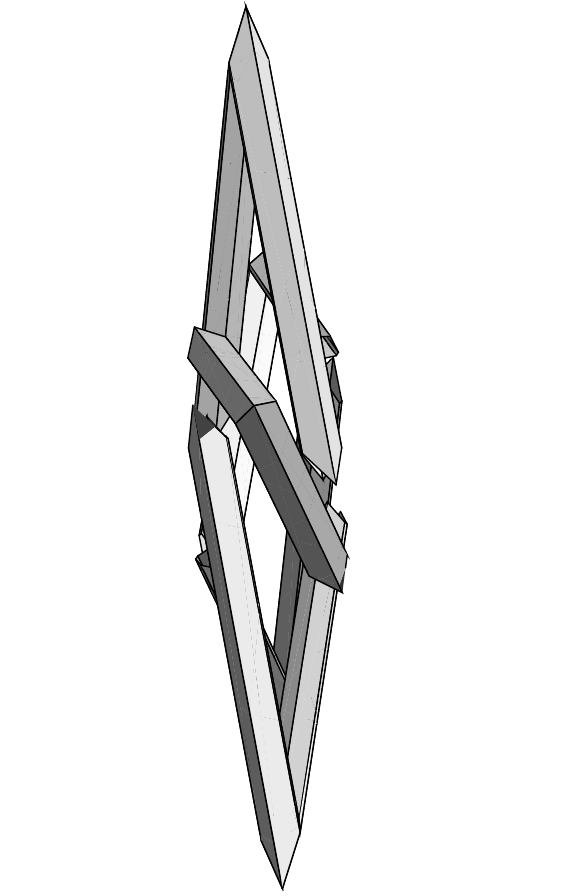}\includegraphics[scale=.3]{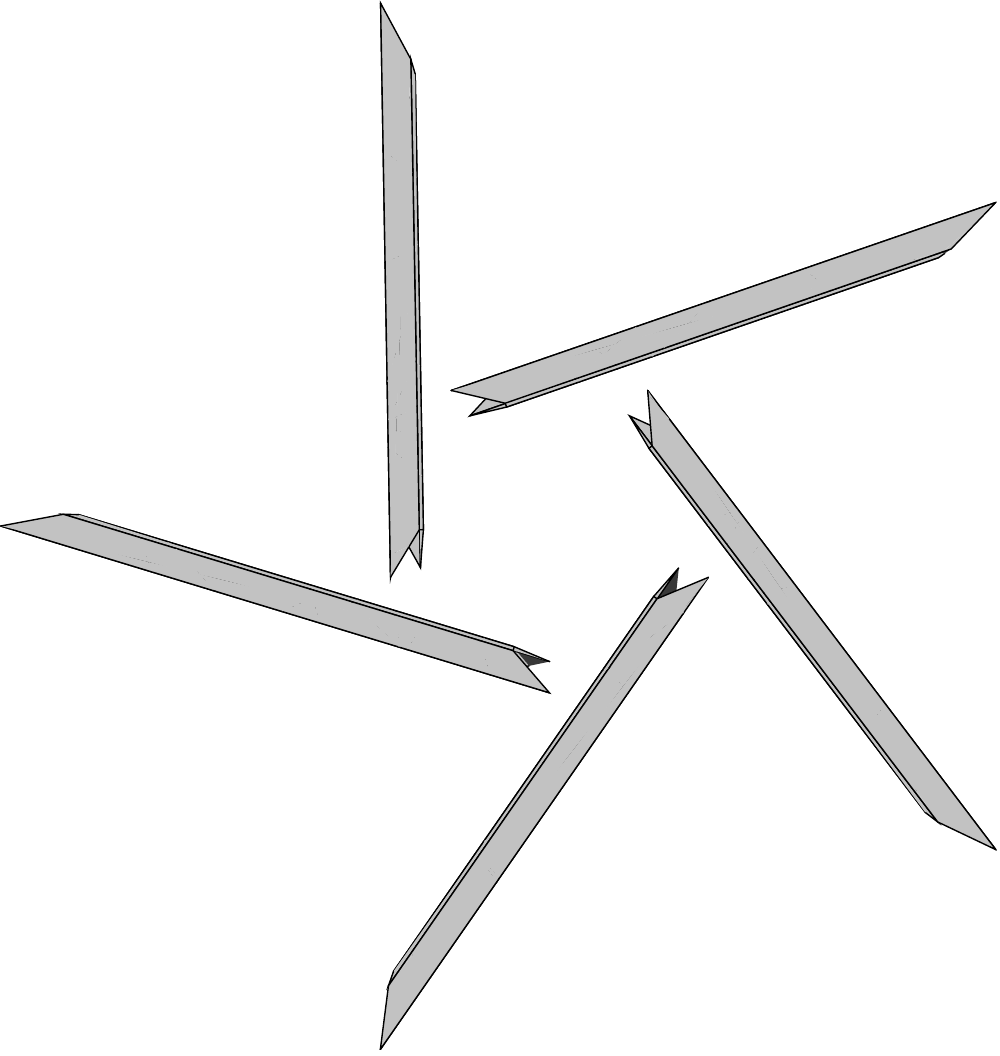} }
\caption{The SIB5P component (left), false visual band in variant B2  (center), falseness of visual band in variant B2 (right), and a single matching (right-er).}
\label{fig:sib5pM}
\end{figure}

However, there is a matching coloring.  The matchings (shown in Figure \ref{fig:sib5pM}) are identical to those of FIET variant A, and so are the combinatorics.

\subsection{Ten Intersecting Bi-3-pyramids without base edges (TIB3P)} Each component of the TIB3P has three 4-sided ``faces."  The TIB3P has icosidodecahedral symmetry, and there are three polypolyhedral variants. (It is denoted 10-3-4 in Lang's nomenclature.) 

The three polypolyhedral variants are substantially different: the components in two (A and B, see Figures \ref{fig:tib3pA} and \ref{fig:tib3pB}) look like glued-together {\bf Y}s, while the components in the third (C, see Figure \ref{fig:tib3pC}) look like sparse whisks.
\begin{figure}[htp]
\centerline{ \includegraphics[scale=.45]{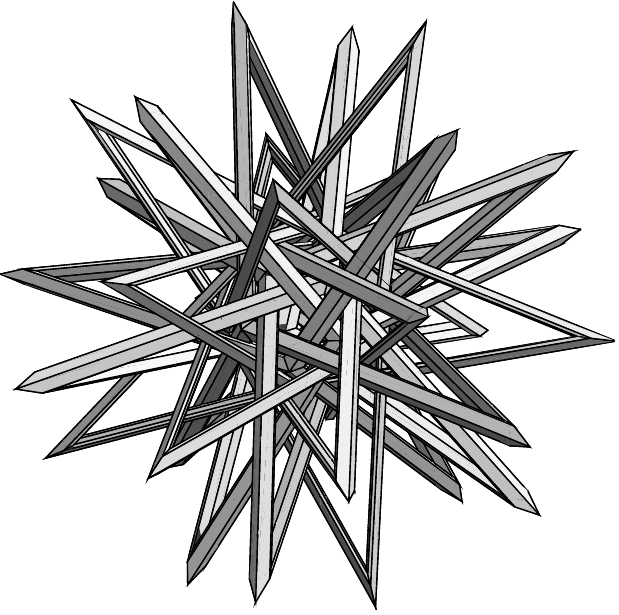} \includegraphics[scale=.35]{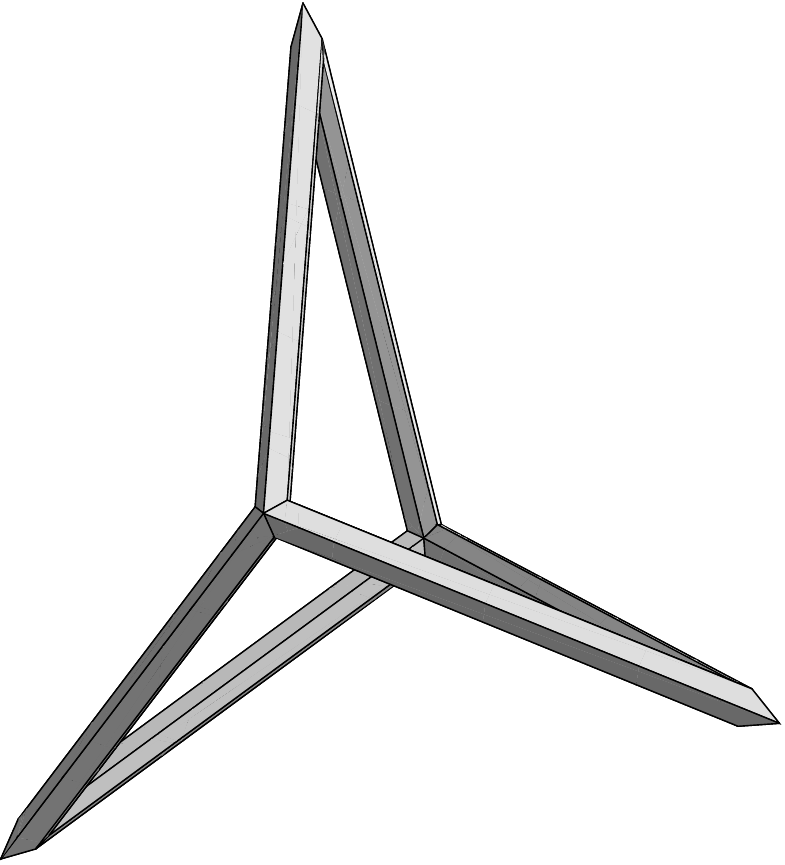}  \includegraphics[scale=.35]{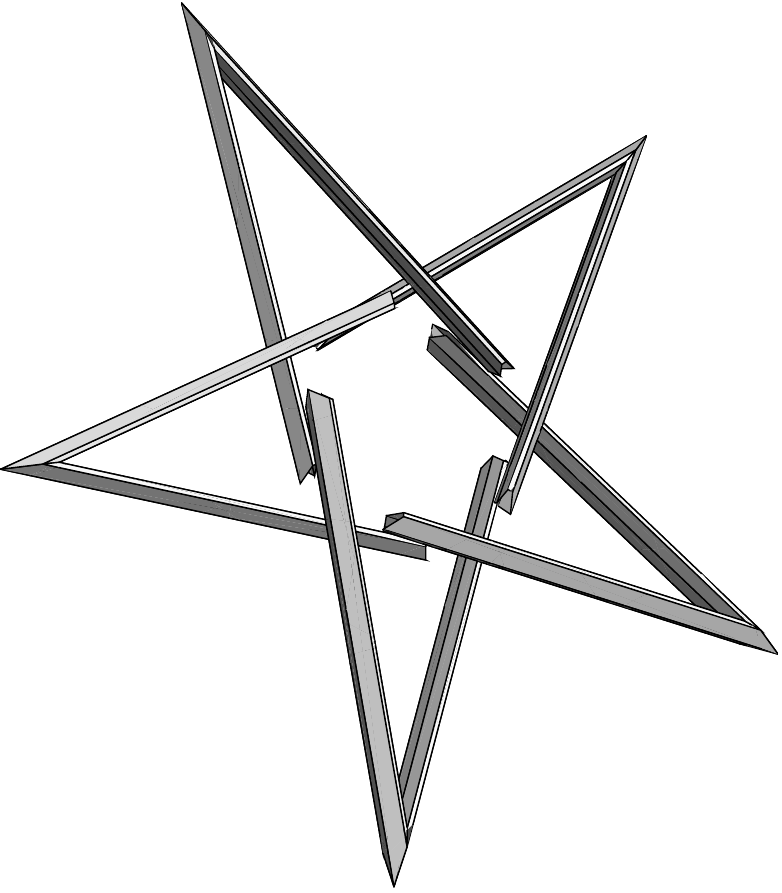}\includegraphics[scale=.35]{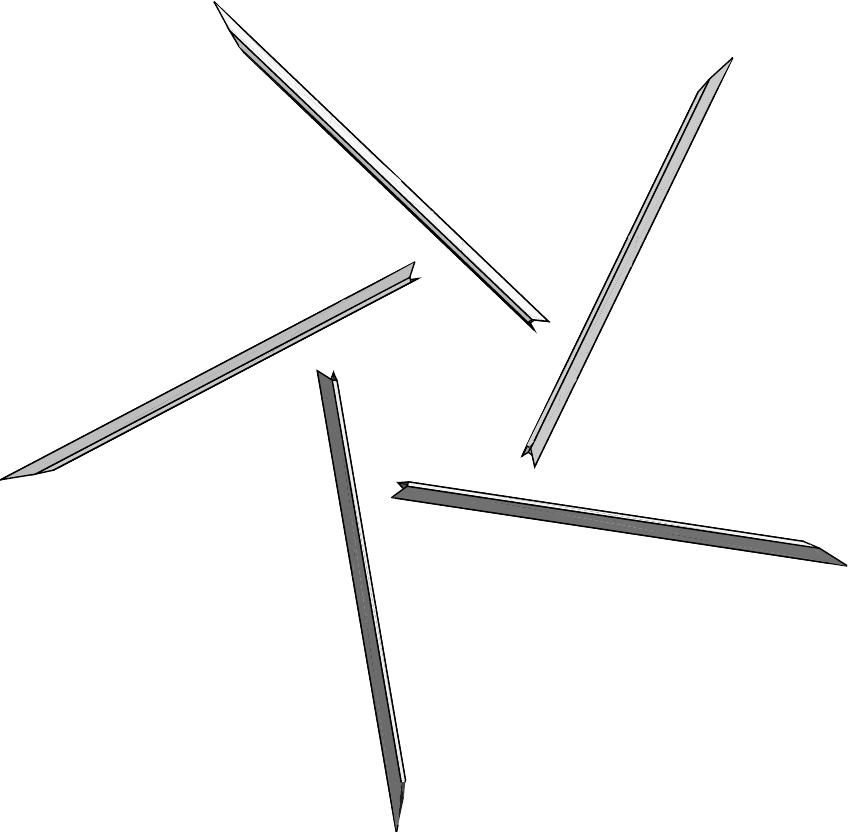} }
\caption{The TIB3P variant A (left), with a single component (center), a visual band (right), and a single matching  (right-er).}
\label{fig:tib3pA}
\end{figure}

\begin{figure}[htp]
\centerline{ \includegraphics[scale=.45]{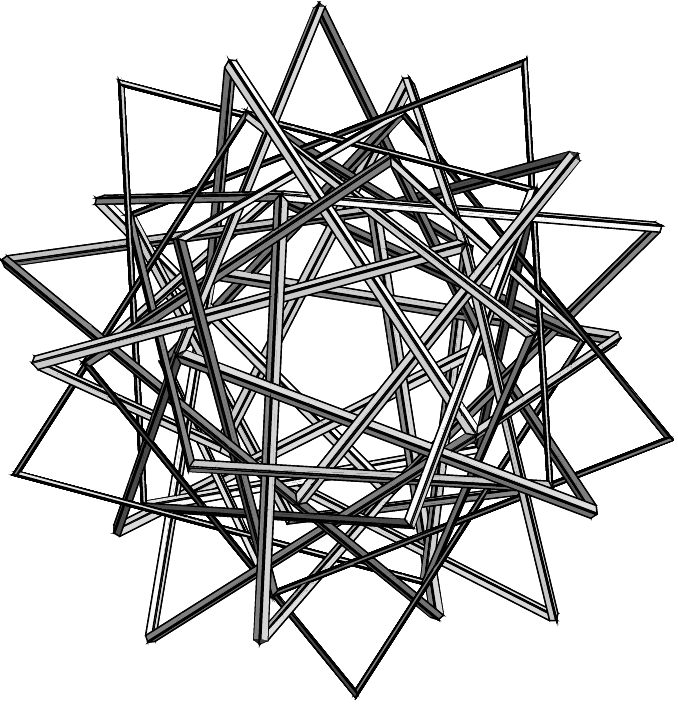} \includegraphics[scale=.35]{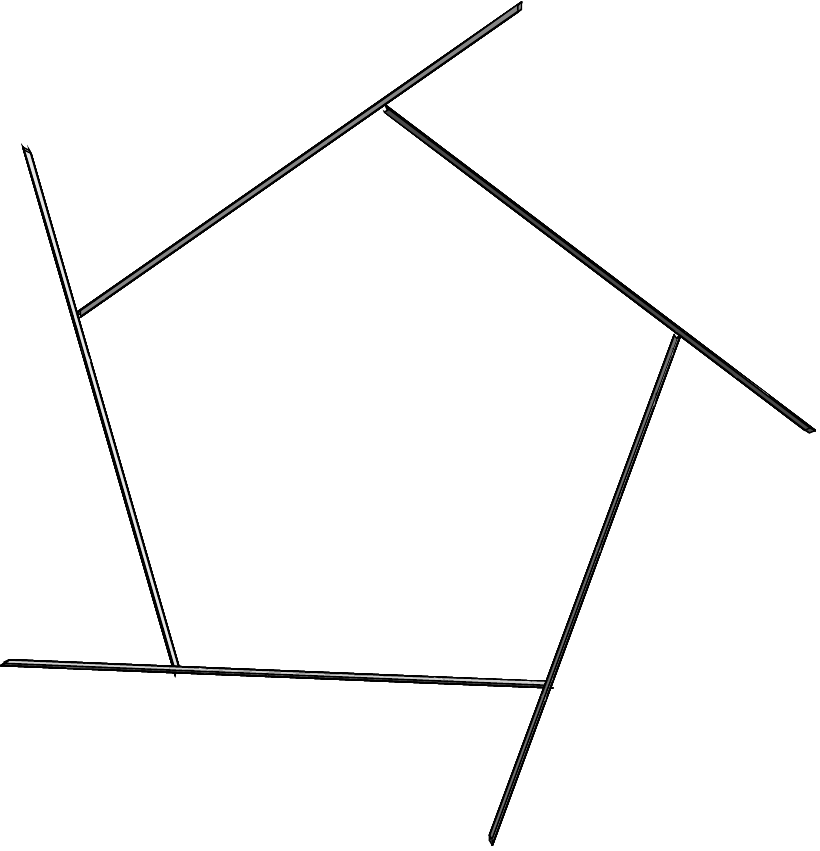} }
\caption{The TIB3P variant B (left), with a single matching  (right).}
\label{fig:tib3pB}
\end{figure}

\begin{figure}[htp]
\centerline{ \includegraphics[scale=.45]{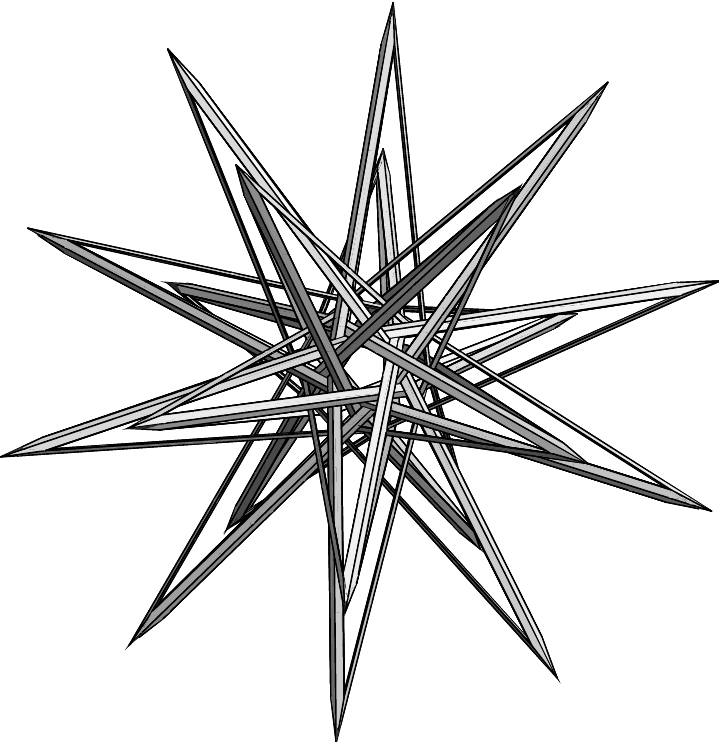}\includegraphics[scale=.4]{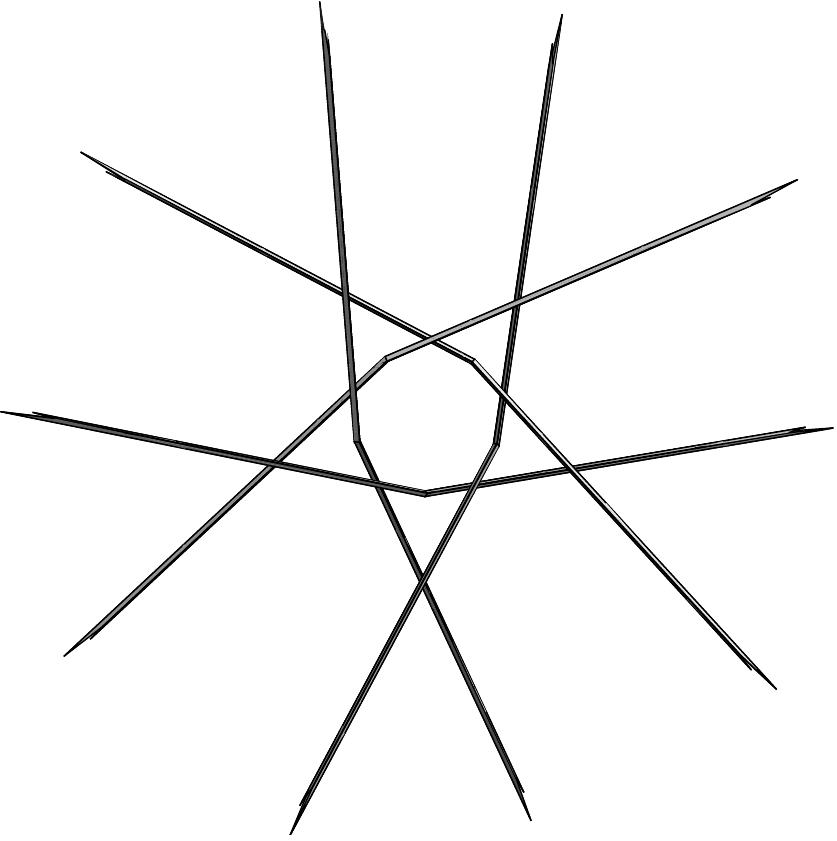}\includegraphics[scale=.25]{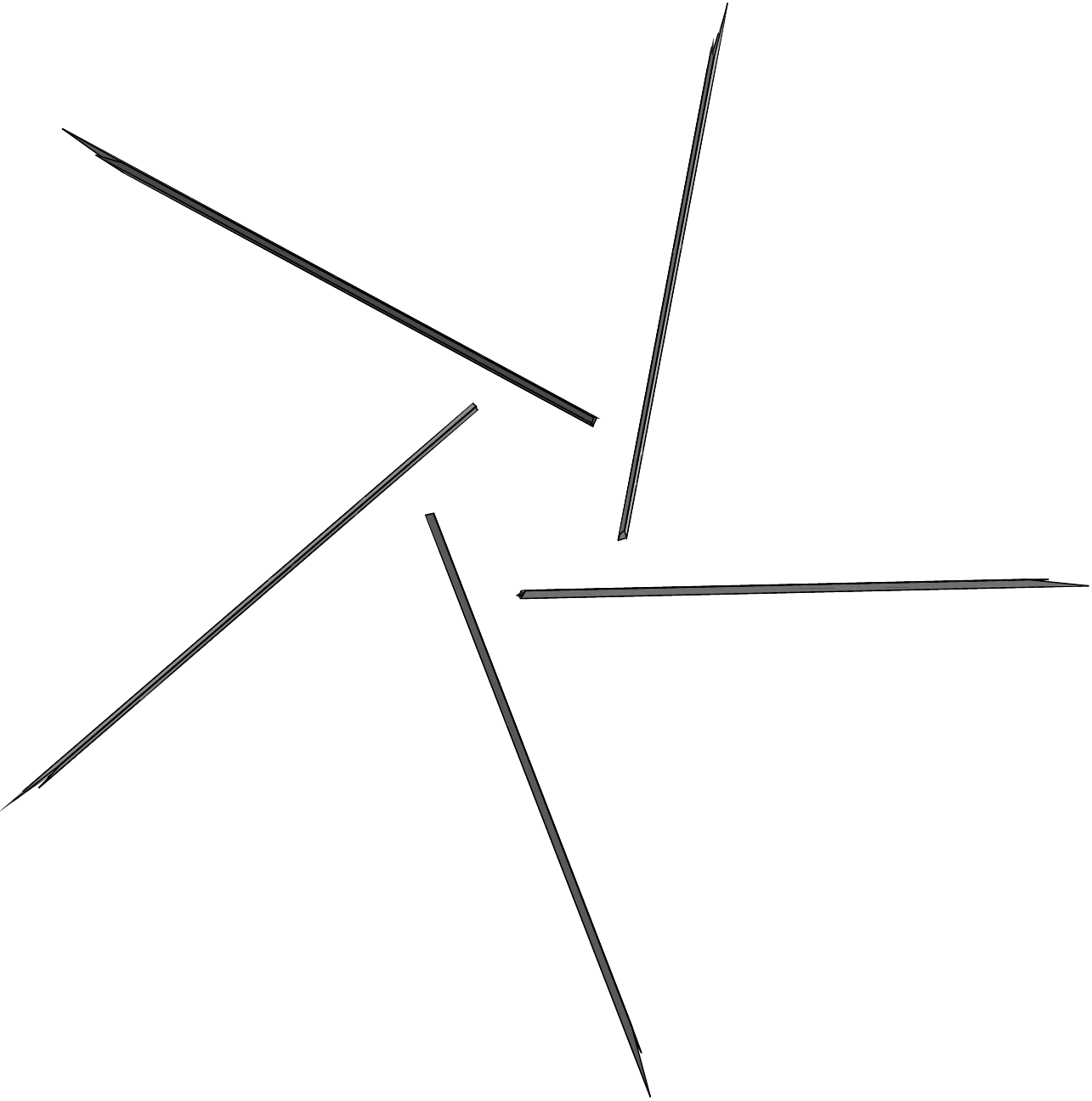} }
\caption{The TIB3P variant A (left), with  a visual band (center), and a single matching  (right).}
\label{fig:tib3pC}
\end{figure}

Variants A and C have visual band colorings (see Figures \ref{fig:tib3pA} and \ref{fig:tib3pC}) , but variant B does not.  The visual bands have  the same combinatorics as the FIET's visual band colorings.  Of course all three variants have matching colorings.  In this case,  the components each use six colors on six edges, and have the same combinatorics as the FIT band colorings; thus there are 12 matching colorings.

\subsection{Four Intersecting Bi-3-pyramids without base edges (FIB3P)} Each component of the FIB3P has three 4-sided ``faces."  The FIB3P has cuboctahedral symmetry, and there are two polypolyhedral variants. (It is denoted 4-3-4  in Lang's nomenclature.) The two polypolyhedral variants are substantially different: the components in one (A, see Figure \ref{fig:fib3pA}) look like glued-together {\bf Y}s, while the components in the other (B, see Figure \ref{fig:fib3pB}) look like sparse whisks.
\begin{figure}[htp]
\centerline{ \includegraphics[scale=.3]{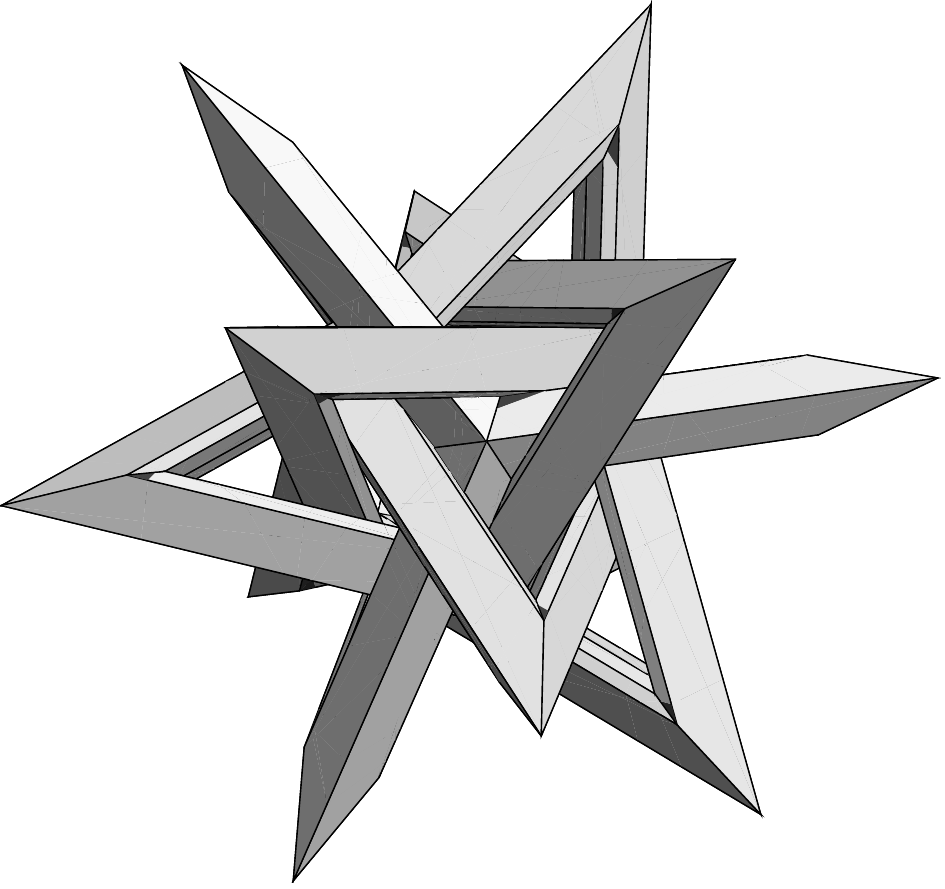} \includegraphics[scale=.3]{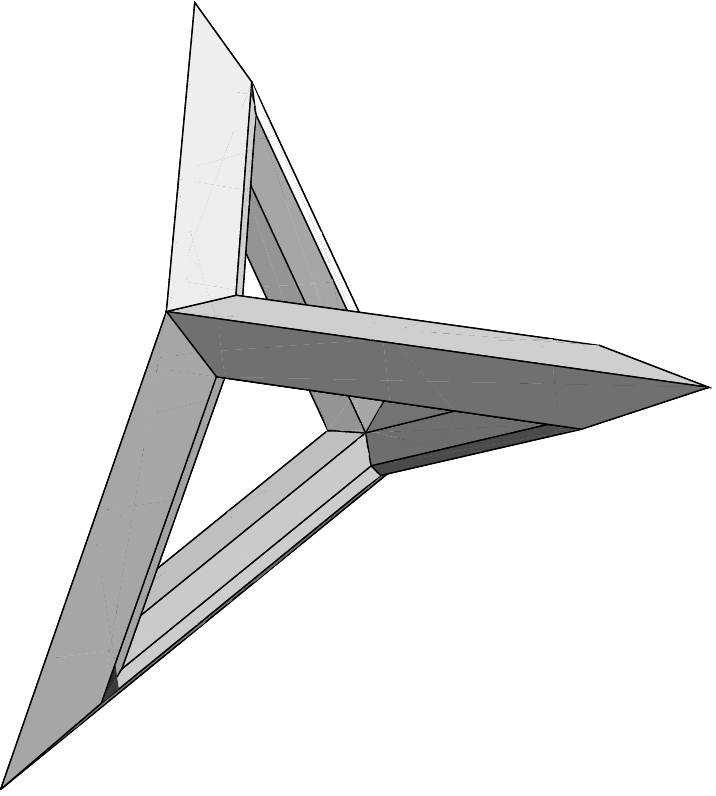}  \includegraphics[scale=.3]{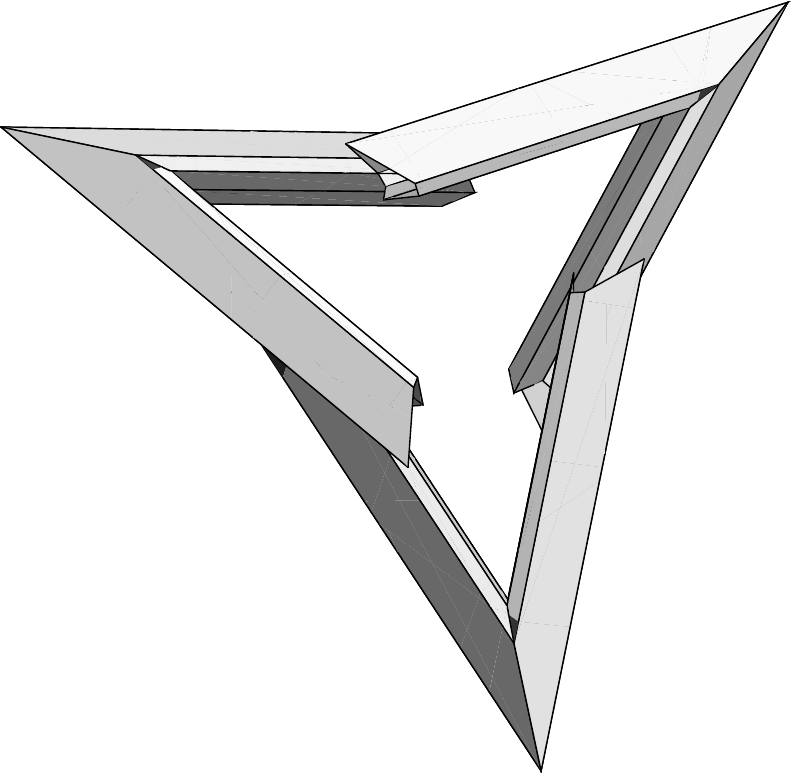}\includegraphics[scale=.3]{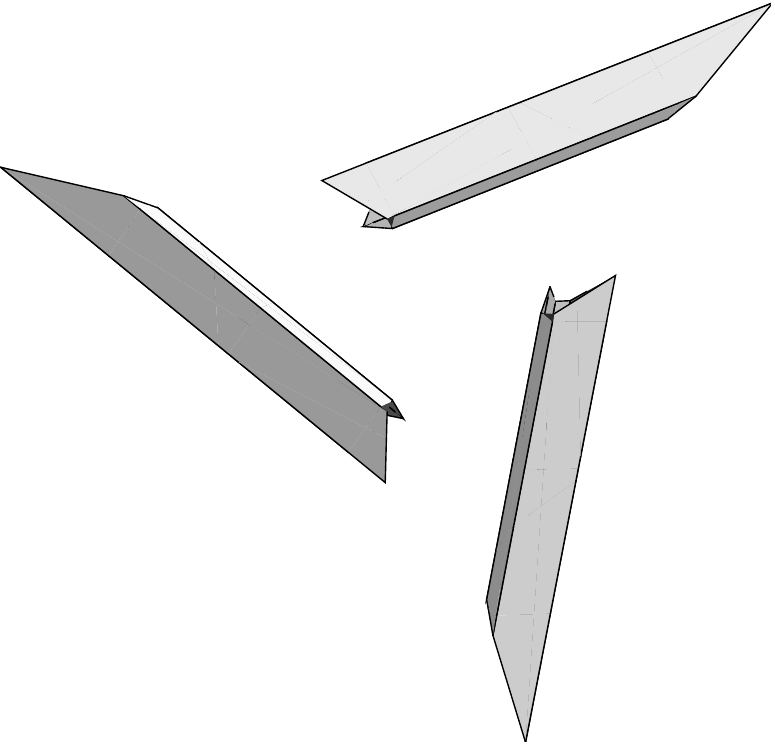} }
\caption{The FIB3P variant A (left), with a single component (center), a band (right), and a single matching (right-er).}
\label{fig:fib3pA}
\end{figure}

Variant A has visual bands as well as matchings (see Figure \ref{fig:fib3pA}).  The visual bands $B$ are composed of 2 consecutive struts from each of 3 components and so correspond to a 3-edge matching $M(B)$ in the symmetry cube. Thus there is only one visual band coloring (see Section \ref{sec:MC}). The matchings come in pairs to form the visual bands.  Following the FIET argument, we have eight matchings and ${8 \choose 4}$ ways of assigning 4 colors to one matching in each pair, and $4!$ ways of assigning the remaining colors, so ${8 \choose 4}\cdot 4!$ matching colorings in total.

In variant B,  
the whisk apices point to the vertices of a bounding cube.  It has a matching coloring (see Figure \ref{fig:fib3pB})    
\begin{figure}[htp]
\centerline{ \includegraphics[scale=.3]{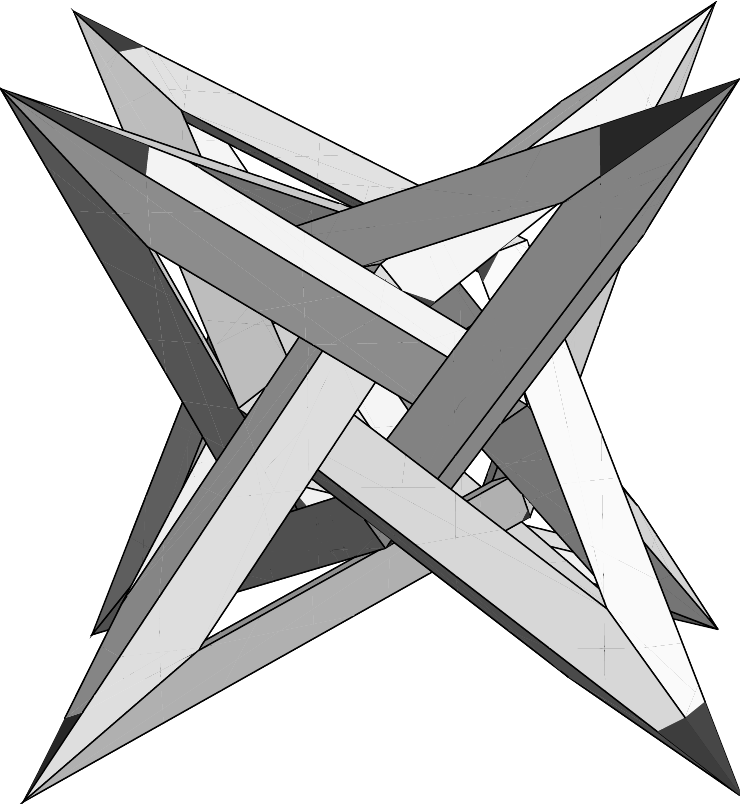} \includegraphics[scale=.23]{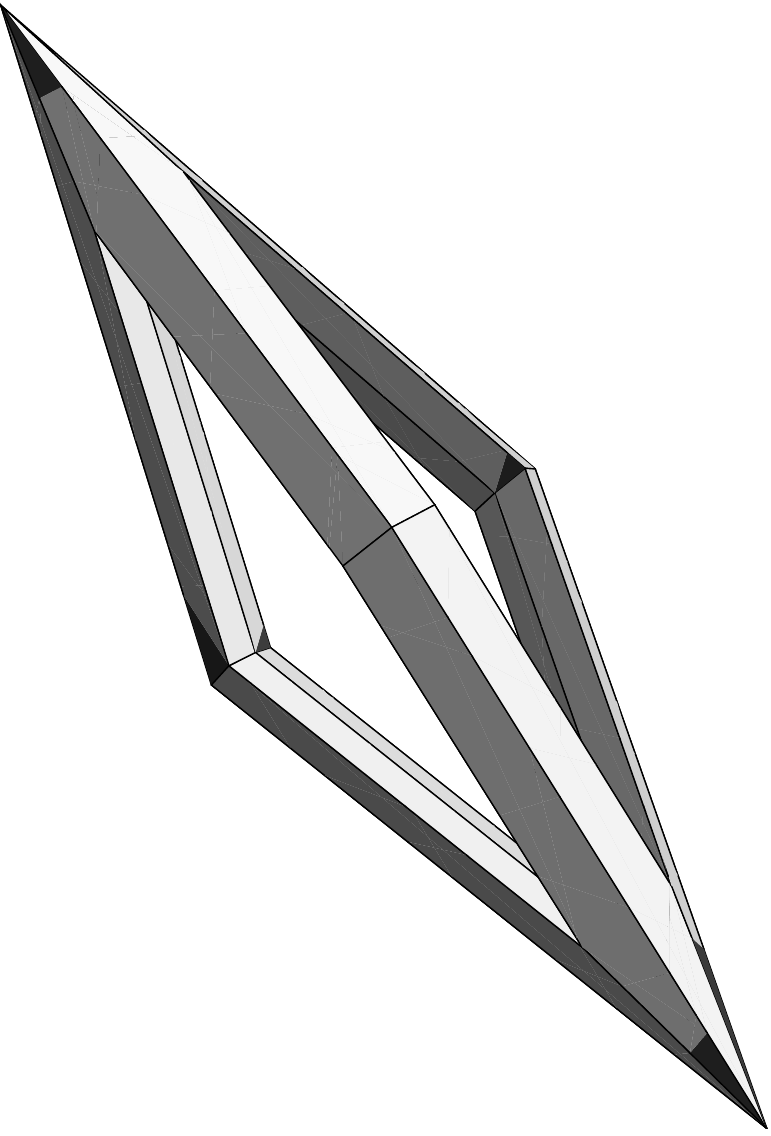}  \includegraphics[scale=.26]{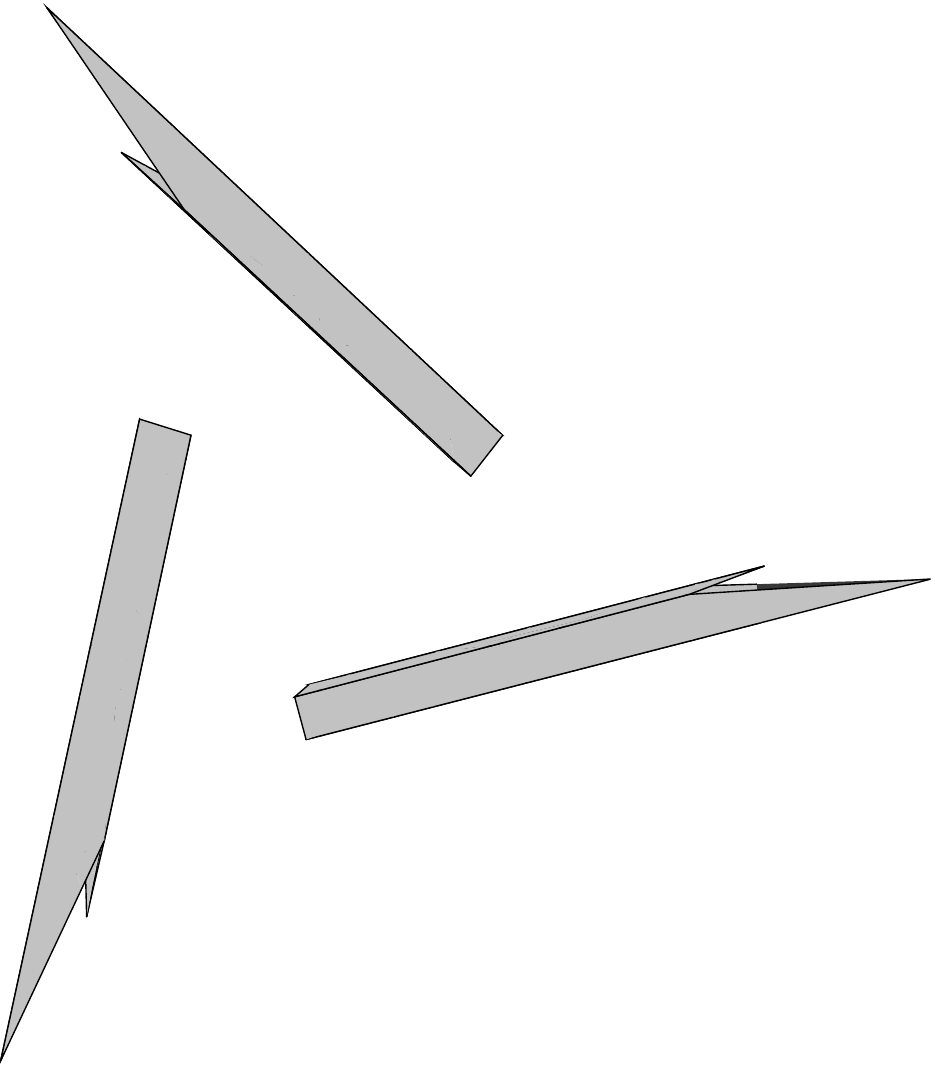} \includegraphics[scale=.23]{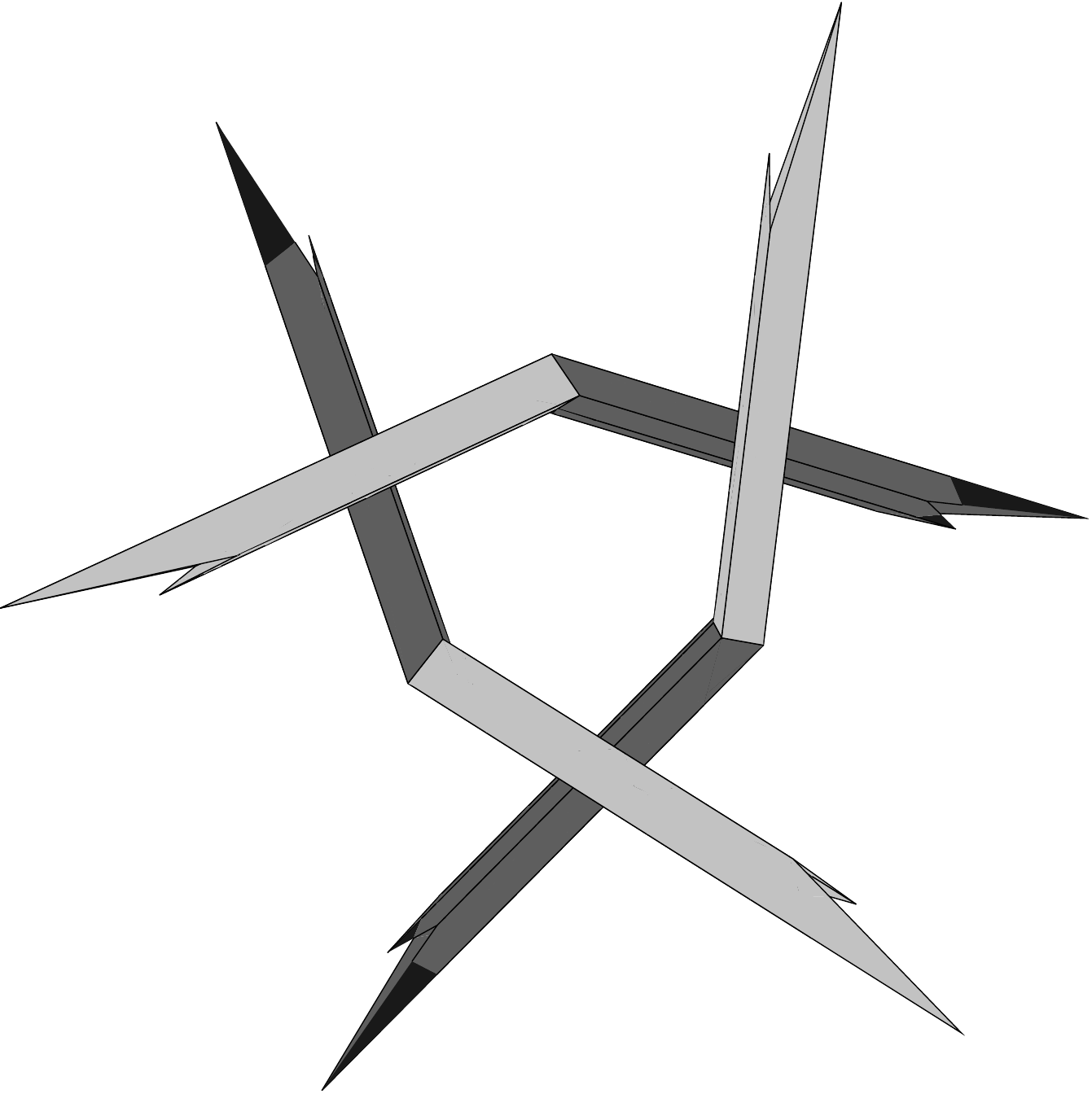} }
\caption{The FIB3PT variant B (left), with a single component (center),  a single matching (right), and a pair of matchings (right-er).}
\label{fig:fib3pB}
\end{figure}
but no visual band coloring, as a visual band coloring would require using two struts of the same color at each 3-valent vertex. Even though there are no visual bands, the matchings come in pairs (see Figure \ref{fig:fib3pB}(right-er)) so the combinatorics are the same as for variant A.

\subsection{Colorings of the polypolyhedra with polygon components}\label{sec:polygons}


When the components of a polypolyhedron are polygons, we easily obtain a visual band coloring where each component is monochromatic (see Section \ref{subsec:visual bands}).  There are other symmetric matching colorings where each component is rainbow colored.  We will not go into the details, but the combinatorics of such examples is exactly the same as other examples already described.

\section{Conclusions and further work}

There are more symmetric colorings of polypolyhedra than described in this paper:

\begin{enumerate}

\item There is a symmetric rainbow coloring of the components of the SIB5P that requires 10 colors.

\item In the FIT there is a non-band set of edges, with \emph{maximal} dot product to a 5-fold symmetry axis, that corresponds to the same dodecahedral matching as a band; this extends to a symmetric coloring of the FIT.

\item The FIT also has a symmetric coloring that corresponds to a decomposition of the dodecahedron into 10 matchings that each have 3-fold symmetry.


\end{enumerate}

Similar analysis as performed above can be done on these symmetric colorings. Other interesting questions abound.  Are there other numbers of colors/matchings that can lead to symmetric colorings of polypolyhedra?  What do the 23 variations of the twenty interlaced triangles have to offer?

There are also relationships between symmetric polypolyhedral colorings.  Looking closely at the symmetric colorings of polypolyhedra with polygon components leads to the discovery of a {\em duality} between polyhedron-component polypolyhedra and polygon-component polypolyhedra.  This will be described in a future paper (stay tuned!).

\bibliographystyle{akpbib}
\bibliography{polypolycolor-6OSME-proceedings.bib}

%
%
%
%
%
%

\end{document}